\documentclass[12pt]{article}

\usepackage{color}
\usepackage{geometry}                % See geometry.pdf to learn the layout options. There are lots.
\usepackage{graphicx}
\usepackage{epstopdf}
\usepackage{graphicx}
\usepackage{amssymb}
\usepackage{amsmath}
\usepackage{amsthm}
\usepackage{amsfonts}
\usepackage{mathrsfs}
\usepackage{stmaryrd}

\newtheorem{theorem}{Theorem}
\newtheorem{lemma}[theorem]{Lemma}

\newtheorem{corollary}[theorem]{Corollary}

\newtheorem{proposition}[theorem]{Proposition}
\theoremstyle{definition}

\theoremstyle{remark}

\numberwithin{equation}{section}

\newcommand\RR{{\mathbb R}}

\newcommand{\vect}[2]{\left( \begin{array}{l} #1 \\ #2 \end{array} \right)}

\def\W2{W^{1,2}({\cal O}(M))}

\newcommand\one{{\mathbf 1}}

\newcommand{\stb}{\mbox{\small $\frac{2}{\sqrt{\beta}}$}}

\newcommand{\tr}{\mathrm{tr}}

\newcommand{\xbolt}[0]{x^{\lightning}}
\newcommand{\ybolt}[0]{y^{\lightning}}

\newcommand{\onex}[0]{x^{(1)}}
\newcommand{\oney}[0]{ y^{(1)}}

\title{Universality of  the Stochastic Bessel Operator} 

\author{
  {Brian Rider}\footnote{Department of Mathematics, Temple University. e-mail: {\tt{brian.rider@temple.edu.}}},  
  { Patrick Waters}\footnote{Department of Mathematics, Temple University. e-mail: {\tt{patrick.waters@temple.edu.}}}
  }

\date{} 
\begin{document}
\maketitle

\begin{abstract} We establish universality at the hard edge for general beta ensembles provided that the background potential $V$ is  a polynomial such that  $x \mapsto V(x^2)$ is uniformly convex and $\beta \ge 1$. The method rests on the corresponding tridiagonal matrix models, showing that their appropriate continuum scaling limit is given by the Stochastic Bessel Operator. As conjectured in \cite{ES} and rigorously established in \cite{RR}, the latter characterizes the hard edge in the case of linear potential 
and all $\beta$ (the classical ``beta-Laguerre" ensembles).

\end{abstract}

%\tableofcontents

\section{Introduction}

We prove a universality result for the limiting distribution of the smallest points for a family of coulomb gas measures.
With any  $\beta >0$ and $a > -1$ these measures are prescribed through the joint densities of $n$ points $\{ \lambda_1, \dots , \lambda_n \}$ on the positive half-line:
\begin{equation}
\label{eigdensity}
   c \prod_{i \neq j} | \lambda_i - \lambda_j|^{\beta}  \prod_{i=1}^n w (\lambda_i), \quad w(\lambda)  = \lambda^{\frac{\beta}{2} (a+1) -1} e^{- {\beta} n  V(\lambda)}.
\end{equation}
In general, $V$ can be any function that is bounded at zero and of suitable growth at infinity; the constant $c = c(V, \beta, a,n)$ is the corresponding normalizer. The particular choice of the weight $w$ is explained by the fact that when $\beta=1,2,4$, $V(x) =x/2$, and $a$ is an integer, \eqref{eigdensity} is precisely the joint density of eigenvalues for the classical Wishart (or Laguerre) ensembles of random matrix theory. These are ensembles of the form $X X^{\dagger}$ for  an $n \times (n+a)$ matrix $X$  of independent real, complex, or quaternion (at $\beta =1, 2,$ or $4$)  mean-zero Gaussians, here normalized to have mean-square $(n \beta)^{-1}$. 

The scaling limit for the smallest points in this and related contexts is now commonly referred to as the hard-edge limit. In the solvable case of complex Gaussian entires ($\beta =2$ and $a=0,1,2, \dots$) a closed form for these distributions was discovered by Tracy and Widom \cite{TW1}, with results for the real and quaternion cases following in \cite{TW2}. Keeping with the Gaussian-type potential ($V(x) =x/2$), but now allowing all possible values of $\beta$ and $a$, the densities \eqref{eigdensity} define the extensively studied ``beta-Laguerre" ensembles. Based on a corresponding tri-diagonal matrix 
model of Dumitriu and Edelman \cite{DE}, Edelman and Sutton \cite{ES} conjectured that the limiting  beta-Laguerre hard edge 
should be described by a certain random differential equation which they tagged the Stochastic Bessel Operator. This was subsequently proved in \cite{RR1}.

The Stochastic Bessel Operator (${\mathrm{SBO_{\beta,a}}}$) takes the form:
\begin{equation}
\label{SB1}
      {\mathrm{SBO_{\beta,a}}} 
      =  - \, e^{(a+1) x + \stb b(x)} \frac{d}{dx}  e^{ -a x - \stb b(x)}  \frac{d}{dx}, 
\end{equation}
where $x \mapsto b(x)$ is a standard Brownian motion. For the present application, this is viewed as acting on 
$L^2( \RR_+, m(dx) )$ for  $m(dx) = e^{-(a+1) x - \stb b(x)} dx)$ with a  Dirichlet boundary condition at the origin. 
Viewed as a random diffusion generator symmetric with respect to $m(dx)$, one sees that ${\mathrm{SBO_{\beta,a}}}$ has almost surely discrete spectrum \cite{RR1}. 

Here we show that ${\mathrm{SBO_{\beta,a}}}$ is a universal object, characterizing the hard-edge scaling limit for $\beta$-ensembles \eqref{eigdensity} with a certain class of polynomial potentials.

\begin{theorem} 
\label{thm:chief}
Consider the  ordered points $0 < \lambda_1 < \lambda_2 < \cdots $ drawn from the $\beta$-ensemble  \eqref{eigdensity} where $V$ is a polynomial such that $x \mapsto V(x^2)$ is uniformly convex and $\beta \ge 1$. 
Denote by $0 < \Lambda_1 < \Lambda_2 < \cdots $ the ordered eigenvalues of  ${\mathrm{SBO_{\beta,a}}}$. Then, 
there is a constant $c_{V, \beta,a}$ such that: as $n \rightarrow \infty$,
\begin{equation}
\label{distscale}
    c_{V, \beta, a} n^2 ( \lambda_1, \lambda_2, \dots)  \Rightarrow (\Lambda_1,  \Lambda_2, \dots ),  
\end{equation}
in the sense of finite dimensional distributions.
\end{theorem}

The restriction to $\beta \ge 1$ is an additional convexity assumption, as we will explain below. Our proof builds on the method in which the ${\mathrm{SBO_{\beta,a}}}$ limit was originally established in the simpler $\beta$-Laguerre setting. 
We identify a family of tridiagonal matrix models which realize \eqref{eigdensity} as their eigenvalue densities, and then demonstrate that ${\mathrm{SBO_{\beta,a}}}$ serves as their appropriate continuum operator limit. Along the way we will see that hard-edge universality is a consequence of a certain functional central limit theorem. The whole program is similar
to the recent soft-edge universality proof using the characterizing Stochastic Airy Operator \cite{KRV}, as we will also explain in greater detail below.

Hard edge universality has previously been addressed at $\beta =1, 2, 4$  via the Riemann-Hilbert Problem method: for $\beta=2$ quite general potentials $V$ are treated in \cite{KV2}, while for $\beta=1$ and $4$ reference \cite{DGKV} considers potentials that are asymptotically monomial. At these  values of $\beta$ the laws \eqref{eigdensity} 
correspond to eigenvalue densities for nonnegative definite matrices $M$ drawn according to the law with density proportion to $(\det M)^{\gamma} e^{-\tr V(M)} d M$ (for choice of $\gamma > -1$). There are also further special values of $\beta$ (outside of $1$, $2$, and $4$) for which the hard-edge of $\beta$-Laguerre can be accessed through multivariate special functions (without appealing to ${\mathrm{SBO_{\beta,a}}}$), see for example \cite{For}.  At the soft edge, besides again $\beta=1,2,4$ results using the Riemann-Hilbert Problem method \cite{DG} and the operator approach of \cite{KRV}, there are the results of Bekerman-Figalli-Guionnet \cite{BFG} and Bourgade-Erd\"os-Yau \cite{BEY3} which hold for a far more general class of potentials. While these 
methods most likely extend to the hard edge, the emphasis here (as in \cite{KRV}) is to demonstrate a simple mechanism for edge universality of random matrices.

\subsection*{Tridiagonals and operator limits}

Let $B = B(x,y)$ denote the $n \times n$ lower bi-diagonal matrix
\begin{equation}
    B_{i,i} = x_i \mbox{ for }  i =1,\dots n, \quad B_{i+1, i} = - y_i \mbox{ for } i = 1, \dots, n-1,
\end{equation}
with the convention that all $x_i$ and $y_i$ are positive. Build the random $B = B(X, Y)$ with  the variables $(X_1, \dots, X_n, Y_1, \dots, Y_{n-1})$ drawn according to the density,
\begin{equation}
\label{thelaw}
 P(x_1, \dots x_n, y_1, \dots, y_{n-1})  = c \, \exp{ \left[- n \beta \tr V ( B B^T )  \right]}  \prod_{k=1}^{n} x_k^{ \beta(k+ a) - 1}  \prod_{k=1}^{n-1} y_k^{\beta k-1},
\end{equation}
on $(\RR_+)^{2n-1}$. Then, the key fact is that the random tridiagonal $B(X,Y) B(X,Y)^T$ has joint eigenvalue density given by \eqref{eigdensity}. This is the general potential analogue of the Edelman-Dumitriu result \cite{DE}. When $V$ is linear \eqref{thelaw} reduces to their representation of $\beta$-Laguerre: all $X_i$ and $Y_i$ are independent with $X_i \sim \frac{1}{\sqrt{n \beta}} \chi_{\beta (i+a)}$ and $Y_i \sim  \frac{1}{\sqrt{n \beta}} \chi_{\beta i}$ ($\chi_r$ denoting a chi variable of parameter $r >0$).
The proof is much the same as that in \cite{DE}, and for completeness is included in the appendix.

Next we recall (from \cite{RR1}) that ${\mathrm{SBO_{\beta,a}}}$ is best understood through its inverse, which has a similar decomposition 
to the the matrix model $(BB^T)^{-1}$. Mapped to act on $[0,1]$ rather than the half line, this inverse takes the form  
$ \mathrm{K}^{T} \mathrm{K}$ in which $\mathrm{K}$ is the integral operator with kernel\footnote{Throughout we use the same notation for any integral operator and its corresponding kernel.} 
\begin{equation} 
\label{SB2}
     \mathrm{K}(s,t) =  \frac{1}{\sqrt{t}}   \left( \frac{s}{t} \right)^{a/2}  \exp \left[ \int_s^t \frac{db_u}{\sqrt{\beta u}} \right]  \mathrm{1}_{s < t}
\end{equation}
on $L^2[0,1]$. 
The strategy that emerges is to show that, after an embedding into $L^2[0,1]$,  $[ n B]^{-1}$ converges to $K$ in a suitably strong sense.

%\begin{equation}
%\label{Bmatrix}
%B(X,Y) =    \left[ \begin{array}{ccccc}
%X_1 &  & & & \\
%-Y_1 & X_2 &  & & \\
%& -Y_2 & \ddots & & \\
%& &  \ddots & X_{n-1} &  \\
%& & & -Y_{n-1}  & X_n  
%\end{array}
%\right]
%\end{equation}

%On the other hand, $B$ has an  explicit (lower triangular) inverse.
%$[B^{-1}]_{ij} =  \frac{1}{X_j} \prod_{k= i }^{ j-1 } \frac{Y_k}{X_k} $ for $i \le j$. 

Now to be completely concrete we specify $V(x) =  \sum_{m=1}^d g_m x^m$, 
and introduce $t \mapsto \phi(t)$ as the unique positive solution of
\begin{equation}
\label{phi_def}
  t = \sum_{m=1}^d m {2m \choose m} g_m \phi(t)^{2m}, \quad  \mbox{ for } t \in [0,1].
\end{equation}
In terms of $\phi$, we also define
\begin{equation}
\label{theta_def}
   \theta(t) = \kappa \left( \int_0^t \frac{du}{\phi(u)} \right)^2, \quad \mbox{ with } \kappa = \kappa_{V, \beta,a} \mbox{ chosen so that } \theta(1) = 1.
\end{equation}
That  $\phi$ has the claimed properties and is such that $\theta$ is finite requires some verification. Granted this however our main technical result is the following.

\begin{theorem} 
\label{thm:main}
Let $x \mapsto V(x^2)$ be a uniformly convex polynomial and take $\beta \ge 1$. 
Denote by $K_n$ the canonical embedding of the random matrices $[ n B(X,Y)]^{-1}$ as operators from $L^2[0,1]$ to itself.
% that is, with basis vectors $e_k = \sqrt{n} \mathrm{1}_{[(k-1)/n, k/n)}$.
% Embed the random bidiagonal matrices $B = B_n$ with entries distributed according to $P$ into $L^2[0,1]$ in the canonical way. 
Then,  for any sequence $n \rightarrow \infty$ there is a subsequence $n' \rightarrow \infty$ and a probability space 
on which $K_{n'}$ converges to  the integral operator $K$ with kernel 
\begin{equation}
\label{newkernel}
K(s,t) = 
\frac{1}{\sqrt{ \phi(s) \phi(t)}} \left( \frac{\theta(s)}{\theta(t)} \right)^{\frac{a}{2} + \frac{1}{4}}  \exp \left[ \frac{1}{\sqrt{\beta}} \int_{\theta(s)}^{\theta(t)} \frac{db_z}{\sqrt{z}} \right] \mathrm{1}_{s < t},
\end{equation}
almost surely in Hilbert-Schmidt norm.
\end{theorem}

One observes  that when $V(x) = x/2$, the definitions \eqref{phi_def} and \eqref{theta_def} yield $\phi(t) = \sqrt{t}$ and $\theta(t) = t$, 
and \eqref{newkernel}  reduces to the advertised kernel $\mathrm{K}$  in \eqref{SB2}. In general we have that
\begin{equation}
\label{var_change}
   \mathrm{K}  ( \theta(s), \theta(t) ) \sqrt{ \theta^{\prime}(s) \theta^{\prime}(t)}  =  2 \kappa  K (s, t), 
 \end{equation}
with $\kappa$ as in \eqref{theta_def}. In other words,  the eigenvalues of $K^T K$ and  $\mathrm{K}^T \mathrm{K}$, defined with the same Brownian motion, agree up to an overall  multiple of $4 \kappa^2$, identifying the $(V, \beta, a)$-dependent scaling constant $c_{V, \beta, a} = 4 \kappa^2 $ in \eqref{distscale}. Here we are using, as is implicit the statement of Theorem \ref{thm:main}, that $K$ and $\mathrm{K}$ are almost-surely Hilbert-Schmidt. To conclude Theorem \ref{thm:chief} is more or less immediate. In the subsequential coupling of Theorem \ref{thm:main} one has $K_n^{T} K_n^{} \rightarrow K K^T$ in trace norm.  Hence, the finite parts of the spectrum of $( n B B^{T})^{-1}$  converge to those of $K^{T} K^{}$ in this manner and so also in distribution. In particular we have as well the convergence in distribution of (any fixed number) of the eigenfunctions as elements of $L^2[0,1]$.

\bigskip

\noindent
{\bf{Remark 1.}}
Rather than embedding $[ n B]^{-1}$ according to the ``flat" basis   $e_k = \sqrt{n} \mathrm{1}_{[(k-1)/n, k/n)}$ and performing the change of  variables \eqref{var_change} after the fact, we could work instead with the suitably weighted $ \mathrm{1}_{[ \theta(k-1/n), \theta(k/n))}$ basis functions to define the embedding. Then, after scaling by $4 \kappa^2$, the corresponding discrete kernels will converge to $\mathrm{K}$ itself.

\bigskip
\noindent
{\bf{Remark 2.}}
The introduced function $\phi$ turns out to provide a first order approximation for the appropriate energy function (Hamiltonian) associated with $P$. It can also be described through
a ``time-dependent" version of the equilibrium measure for the eigenvalue law \eqref{eigdensity}. In particular, let $V_t  = t^{-1} V$ for $t \in (0,1]$ and consider
$$
   \mu_t = \mbox{argmin}_{\mu \in M} \int_0^{\infty}  V_t(x)  \mu(dx) - \int_0^{\infty}  \int_0^{\infty} \log | x - y | \mu(dx) \mu(dy),
$$
where $M$ is the space of probability measures of the half-line. It is the case that $\mu_t$ has support $[0, \phi(t)]$.

\subsection*{Overview of the proof}

Using the explicit inversion formula for bidiagonal matrices, the basic object of study is now understood to be the random kernel operator
\begin{equation}
\label{discreteK}
 K_n ( s, t) = \frac{1}{X_{j}} \prod_{k= i }^{ j -1 } \frac{Y_k}{X_k} \,  \mathrm{1}_{\Gamma_ {i j}}(s,t). 
\end{equation}
Here
    $\Gamma_{ij}$ is the set on which  $s \in \left[  \frac{i-1}{n} , \frac{i}{n} \right)$,
    $  t \in \left[ \frac{j-1}{n},  \frac{j}{n} \right)$, and $s < t$, 
when $i=j$ the product in \eqref{discreteK} is understood to equal one. Given this expression, 
that $\mathrm{spec} ( [n B B^T]^{-1})$  $= \mathrm{spec} (K_n^T K_n^{})$ can be checked by hand.

 The measure $P$ under which $K_n$ is drawn has the form $\frac{1}{Z} e^{-n \beta H} dx dy$ with Hamiltonian 
\begin{equation}
\label{theH}
H(x,y) = \tr \,  V{\left( B B^T \right)} - \sum_{k=1}^n  \left( \frac{k}{n} + \frac{a}{n} - \frac{1}{n \beta} \right) \log x_k 
  - \sum_{k=1}^{n-1} \left( \frac{k}{n} - \frac{1}{n \beta} \right) \log y_k.
\end{equation}
Our assumptions imply that $P$ is uniformly log-concave, that is $ (\nabla^2  H)(x,y) \ge c I$ for some $c >0$ and all $(x,y) \in \RR_+^{2n-1}$. In  particular, with $\beta \ge 1$ each of the log terms in \eqref{theH} has nonnegative second derivative.
One then concludes by noting that,
$$
 \tr V( B B^T) = \frac{1}{2} \,  \tr  V( A^2), \quad \mbox{ for } A =  \left( \begin{array}{cc} 0 & B \\ B^T & 0 \end{array} \right).
$$
and applying C.~Davis' theorem \cite{CD}: a (uniformly) convex function of a Hermitian matrix is (uniformly) convex as a function of its entries. Since \eqref{discreteK} is a simple functional of the process $k \mapsto (X_k, Y_k)$, one is left to quantify the anticipated Gaussian fluctuations of $(X_k, Y_k)$ about the minimizer $(x_k^o, y_k^o)$ of the Hamiltonian $H$.

In Section \ref{sec:Min} we develop a fine (out to $o(n^{-1})$) approximation of the minimizer which allows us to establish the correct centering:
\begin{align}
\label{mean_convergence}
  \lim_{n \rightarrow \infty} 
  \sum_{k= \lfloor n s \rfloor}^{\lfloor nt \rfloor} \log \frac{x_k^o}{y_k^o} 
   =  &  - \left( \frac{a}{2} + \frac{1}{4} \right) \log \frac{\theta(t)}{\theta(s)} +  \frac{1}{2} \log \frac{\phi(t)}{\phi(s)}, 
   %\\  &  \mbox{ for all } s, t \mbox{ with } 0 < s < t < 1.   \nonumber
\end{align}
for all fixed $s, t$ with  $0 < s < t < 1$.   Granted this, the limiting kernel \eqref{newkernel} is identified by showing that
\begin{equation}
\label{path_convergence}
   X_{\lfloor nt \rfloor} \Rightarrow \phi(t), \quad   \quad
   \sum_{k= \lfloor n t \rfloor}^n \log \frac{X_k/x_k^o}{Y_k/y_k^o} \Rightarrow \frac{1}{\sqrt{\beta}} \int_{\theta(t)}^1 \frac{d b_u}{\sqrt{u}}
\end{equation}
in the Skorohod topology on $(0,1]$. Here the polynomial assumption on $V$ is important as it gives $P$ a Markov field property: for example, $(X_i, Y_i)$ and $(X_j, Y_j)$ with $|i - j| > d$ are conditionally independent 
given any intervening block of variables of length $d$.  The implied decorrelation is quantified in a deterministic way,  by showing a decay of dependence of the minimizers of conditional versions of the Hamiltonian $H$ with respect to boundary conditions. These estimates also appear in Section \ref{sec:Min}. Section \ref{sec:Gauss} builds up further properties of the measure $P$, in particular demonstrating Gaussian concentration about the minimizer in terms of tail properties as well as Gaussian approximation of the expectation of various test functions. Given all this, the proof of \eqref{path_convergence} appears in Section \ref{sec:CLT}.

Together \eqref{mean_convergence} and \eqref{path_convergence} provide point-wise convergence (in law) of 
$K_n$ to $K$. To prove that  $ \int_0^1 \int_0^1 | K_n - K|^2  \rightarrow 0$ (over subsequences) as claimed in Theorem \ref{thm:main} requires a certain domination of $K_n$ by a tight family of  $L^2([0,1]^2)$ kernels. This is carried out in Section  \ref{sec:Norm}.  

\subsection*{Comparison with the soft edge}

Before getting on with it, we have a few comments on the technical differences between the present result and the program carried out in \cite{KRV} for the soft edge. While the Stochastic Airy Operator is a more delicate object than our integral operator $K$, being a differential operator and so ``local", to understand the underlying operator convergence of the tridiagonal models one only requires fine information on the first $O(n^{1/3})$ entries of those matrices. For the hard edge we have in a sense to resolve $O(n)$ of $n$ variables. This requires a much more elaborate estimate on the minimizers of $H$, as well as a better handle on the decorrelation between separate stretches of variables under $P$. While our functional central limit theorem in Section \ref{sec:CLT} follows a standard blocking strategy, the corresponding calculation in \cite{RRV} is really a ``one-block" estimate. This also in part explains why the allied method in \cite{KRV} works for all $\beta > 0$. For us the issue is near ``the singularity", or for $X$ and $Y$ of small index where the measure $P$ becomes less coercive. Or, said another way: where, when $\beta < 1$, the Hamiltonian fails to be convex.  While the same issue appears in the soft edge, the troublesome indices are beyond the $O(n^{1/3})$ cutoff and one can get by with rather rough estimates on that part of the field. Again, for the present calculation we simply need more precise control of these entries (as evidenced in particular in the operator norm estimates of Section  \ref{sec:Norm}). Assuming $\beta \ge 1$ covers the classical cases, and makes an already technical paper a little less so.

\section{Minimizers}
\label{sec:Min}

As indicated above, the function $\phi$ introduced in \eqref{phi_def} serves as a first order approximation to the minimizer $(x^o, y^o)$ of the Hamiltonian $H$. The idea is to focus about a fixed index $k$ at a continuum position $t = k/n$, ignoring the additional 
 $(a/n-1/n \beta)$ and $(-1/n\beta)$  multipliers of the $\log x_k$ and $\log y_k$ terms in \eqref{theH}. Then from the potential
we keep only those terms in $\tr \, V(M M^T)$  which involve $(x_k, y_k)$. Assuming that  
the minimizer is locally constant ({\em{i.e.}}, we posit all $x_\ell$ and $y_\ell$ for $|\ell-k | \le d$ equal some $x$ and $y$), we arrive at the following ``coarse Hamiltonian" (at fixed $t \in [0,1]$):
\begin{equation}
\label{coarseH}
H_t(x,y) = \sum_{m=1}^d  g_m \sum_{\ell=0}^m {m \choose \ell}^2 x^{2 \ell} y^{2m-2 \ell} - t \log x - t \log y. 
\end{equation}
Plainly $H_t$ is symmetric in $x $ and $y$, its common ``coarse minimizer"  $x(t)=y(t)$ defining $\phi(t)$. Note that the relation
\eqref{phi_def} is simply the equation for a critical point of $H_t$.

A similar approximation was employed in \cite{KRV} where the analogous objects are referred to as the
local Hamiltonian and corresponding local minimizer.  Here though we require a sharper approximation. In particular, 
to pin down the limiting mean of the log potential, recall \eqref{mean_convergence}, one must refine $(x^o, y^o)$ to 
$o(n^{-1})$ errors.
%\footnote{As will be seen in passing the error between the coordinates of $(x^o,y^o)$ and $\phi$ is  $O(n^{-1})$ in the bulk.} For this purpose we introduce the following.

\bigskip

\noindent
{\em{Definition}: } In terms of $\phi$ define the functions $t \mapsto \onex(t)$ and  $t \mapsto \oney(t)$ via:
\begin{align}
\label{finemin2}
\onex(t) - \oney(t) =& \left(a+\frac{1}{2}\right) \left(\int_{0}^t \frac{du}{\phi(u)}\right)^{-1} - \frac{\phi'(t)}{2}, 
 \nonumber \\
\onex(t) + \oney(t) =& \left( a - \frac{2}{ \beta} \right)\phi'(t). 
%\nonumber
\end{align}
Then,  for $i \in [1, n]$, set 
\begin{equation}
\label{finemin1}
\xbolt_i = \phi(i/n)+\frac{\onex (i/n) }{n}, \quad \quad
\ybolt_i = \phi(i/n)+\frac{\oney (i/n)}{n}.
\end{equation}
We will refer to $(\xbolt, \ybolt)$ as the ``fine minimizer".

\bigskip

Proposition \ref{prop:FineMin} proved below in this section shows that for bulk indices  $| x_i^o - \xbolt_i |$ and $|y_i^o - \ybolt_i|$ are in fact $O(n^{-2})$, from which the desired appraisal \eqref{mean_convergence} follows.

The identification of $(\xbolt, \ybolt)$  relies on uniform 
convexity in an essential way.  Indeed, a characterization equivalent to
$\mbox{Hess} \, H(x,y) \ge c I$ for all $(x,y) \in \RR_{+}^{2n-1}$ is that
\begin{equation}
    c \| (x,y) - (x', y') \|_2^2 \le   \Big\langle  \nabla H(x,y) - \nabla H(x', y') , \, (x,y) - (x', y') \Big\rangle,
\label{ConvexGift1}
\end{equation}
for all $(x,y)$ and $(x', y')$. Putting $(x', y') = (x^o, y^o)$, the undetermined true minimizer, and applying the  
Cauchy-Schwartz inequality \eqref{ConvexGift1} implies that
\begin{equation}
\label{ConvexGift2} 
\| (x,y) - (x^o, y^o) \|_2 \le \frac{1}{c} \| \nabla H(x,y) \|_2.
\end{equation}
The point is that the fine minimizer $(\xbolt, \ybolt)$ has been engineered so that  $\nabla H(\xbolt, \ybolt)$ vanishes to sufficiently high order.

\subsection{Identifying the fine minimizer}

The goal of this subsection is to establish the following.

\begin{lemma} 
\label{lem:GradHBound}
For any $i \in [1, n-d)$ it holds that
\begin{equation}
\label{eq:GradH}
   \left| \frac{\partial}{ \partial x_i} H(\xbolt, \ybolt) \right| +  \left| \frac{\partial}{\partial y_i} H(\xbolt, \ybolt) \right|
   \le c \frac{1}{\sqrt{n i^3}}
\end{equation}
with a constant $c$ depending on $V, \beta$ and $a$.  For $i \in [n-d, n]$ the right hand side of 
\eqref{eq:GradH} can be replaced by $O(1)$.
\end{lemma}

Before this however we go back and verify that the coarse minimizer has the various properties claimed above, and also 
provide an estimate on its shape near the singularity which will be needed for the proof of Lemma \ref{lem:GradHBound}.

\begin{lemma} The coarse minimizer $\phi(t)$ is unique, positive and increasing, and satisfies
\label{phi_bounds}
\begin{equation}
\label{eq:phi_bounds}
   c^{-1} \le  \phi(t){t^{-1/2}}, \  \phi'(t) {t}^{1/2} , \ \phi''(t) t^{3/2}  \le c
\end{equation}
for a constant $c$ and all small enough $t > 0$.
\end{lemma}

\begin{proof} Uniqueness can be seen from the following alternative description of $H_t$.  Let $C$ be the circulant version of our bidiagonal matrix $B$ in variables $x_1, \dots x_m$ and $y_1, \dots, y_m $ (with $m > d$) and consider minimizing
\begin{equation}
\label{IntrocoarseH}
  (x,y) \mapsto \tr \, V (C C^T) - t \sum_{k=1}^m (\log x_k + \log y_k).
\end{equation}
By another application of C.~Davis' theorem, this is a convex function. It is also invariant under rotations of the indices, and so its minimizer satisfies $x_k \equiv x$ and $y_k  \equiv y$ for some  $x$ and $y$ and all $k$. But making this substitution one finds that the right hand side of \eqref{IntrocoarseH} equals $m H_t(x,y)$.

Showing that $t \mapsto \phi(t)$ is positive and increasing comes down to  showing  that the right hand side of  \eqref{phi_def}
%,$\phi \mapsto  \sum_{m=1}^d m {2m \choose m} g_m \phi^{2m}$,
is increasing as a function of $\phi$. Rewrite that expression as in
\begin{equation}
\label{IntV}
  \sum_{m=1}^d m {2m \choose m} g_m \phi^{2m}
     =  \phi \int_{0}^2 2 u \phi V'( \phi^2 u^2)  \frac{ 4 u du }{\sqrt{ 4- u^2}}.
 % = \phi \int_{-2}^2 2 u \phi V'( \phi^2 u^2) \,  u \sqrt{ \frac{2+u}{2-u} } \, \frac{du}{ 2 \pi}.
\end{equation}
Since $x \mapsto V(x^2)$ is uniformly convex, $ x \mapsto 2 x V'(x^2)$ is increasing. The integral on the right hand side of \eqref{IntV} is therefore  a weighted average of such increasing functions, which yields the claim.

As for  \eqref{eq:phi_bounds}, again by uniform convexity we have that $g_1 >0$, and so for small $t$ 
the relation \eqref{phi_def} takes the form $t = 2 g_1 \phi^2 (1+o(1))$. This shows that $\phi(t)$ is bounded above and below by a multiple of $\sqrt{t}$ for small $t$.  The estimates on $\phi'$ and $\phi''$ follow suit by considering successive derivatives of  \eqref{eq:phi_bounds}.
\end{proof}

\begin{proof}[Proof of Lemma \ref{lem:GradHBound}] 
The starting point is a lattice path representation for the diagonal entries of powers of $BB^T$:  for $i \in (d, n-d)$,
\begin{align}
[(BB^T)^m]_{ii} =& \sum_{p \in P_{m}} \prod_{j=1}^m 
\left( \begin{cases} x_{i+p(2j-1)} & \text{ if }p(2j)=p(2j-1) \\ -y_{i+p(2j-1)-1} & \text{ if } p(2j)=p(2j-1)-1 \end{cases} \right) \nonumber  \\
& \hspace{60pt}
\times \left( \begin{cases} x_{i+ p(2j)} & \text{ if }p(2j+1)=p(2j) \\ -y_{i+p(2j)} & \text{ if } p(2j+1)=p(2j)+1 \end{cases} \right).\label{LatticePaths}
\end{align}
Here $P_m$ denotes the collection of random walk paths of length $2m$ beginning and ending at height $0$ and 
constrained as follows. At odd-timed steps $j$ (corresponding to selecting an entry from $B$), the path can either take a step of type $\rightarrow$ and remain at $p(j)$ or take a step of type $\searrow$ and $p(j+1) = p(j)-1$. At even-timed steps $j$ 
(corresponding to selecting an entry from $B^T$) the step can either be again of type $\rightarrow$ or of type $\nearrow$ , in which case 
$p(j+1) = p(j)+1$.

Note that for $i \in [1,d] \cup [n-d, n]$ certain paths will be truncated, resulting in a more cumbersome expression.

From \eqref{LatticePaths} it is easy to see that: with again $i \in (d, n-d)$,
\begin{align}
\label{DerivFormula}
 \frac{\partial }{\partial x_i}\tr V(BB^T)  =&  
\sum_{m=1}^d g_m \sum_{r=1}^{2m} \frac{r}{ x_{i}} \sum_{p\in P_{m,r}}\prod_{j=1}^m 
\Bigg[\left( \begin{cases} x_{i+ p(2j-1)} & \text{ if step $2j-1$ is $\rightarrow$} \\ y_{i+p(2j-1)-1} & \text{ if step $2j-1$ is $\searrow$}\end{cases}\right) \nonumber \\
& \hspace{.32\textwidth} \times \left( \begin{cases} x_{i+p(2j)} & \text{ if step $2j$ is $\rightarrow$} \\ y_{i+p(2j)} & \text{ if step $2j$ is $\nearrow$}\end{cases}\right)\Bigg],  
\end{align}
where now $P_{m, r}$ is the subset of $P_m$ comprised of those paths that posses 
exactly $r$ steps of type $\rightarrow$ at height zero.

The next step is to substitute the values of the fine minimizer into \eqref{DerivFormula}. These are used in the form: 
 for $|k| \le d$ and $i+k \in [1,n]$,
\begin{equation}
\label{shift_fine}
  \xbolt(i+k) =  \phi(i/n) + \frac{1}{n} \Bigl( k \phi'(i/n) + \onex(i/n) \Bigr) + O \left(  \frac{1}{\sqrt{n i^3} } \right),
\end{equation} 
with a like expression for $\ybolt(i+k)$. To see \eqref{shift_fine},  $\phi(t +k/n)$ and $\onex(t+k/n)$ are expanded out to
second and first order, respectively. That both $(d^2/dt^2) \phi(t)$ and $(d/dt) \onex(t)$ are $O(t^{-3/2})$ follows from Lemma
\ref{phi_bounds}. The result of the substitution is:
\begin{align}
\label{DerivFormula2}
& \frac{\partial}{\partial x_i}  \tr V(BB^T) ( \xbolt,\ybolt)  \\
&= \sum_{m=1}^d g_m  \left[  A_{m} \phi^{2m-1} + \frac{\phi^{2m-2} }{n}\left(B_{m}  \onex+ C_{m}  \oney +D_{m} \phi' \right)\right] + O\left(\frac{1}{\sqrt{n i^{3}}}\right), \nonumber
\end{align}
where the functions $\phi,\phi',\onex$ and $\oney$ are all evaluated at $i/n$,
and
\begin{alignat}{2}
\label{coefficients}
A_m =& m\binom{2m}{m}, & \qquad \qquad 
B_m=& \frac{2m^2-2m+1}{2m-1} %m\binom{2m}{m}
A_m,  \\
C_m =& \frac{2m^2-2m}{2m-1}  A_m, &\qquad\qquad
D_m =& -\frac{m^2-m}{2m-1} A_m.  \nonumber
\end{alignat}
Putting off the derivations behind \eqref{coefficients} we can complete the proof.

For the $x_i$-derivative of the logarithmic term in the Hamiltonian  we have that,
\begin{equation}
\label{logderiv}
 \Bigl( \frac{i+ a - \beta^{-1}}{n}  \Bigr) \frac{1}{\xbolt(i)} = \frac{i + a - \beta^{-1}}{n \phi(i/n)}  - \frac{i \onex(i/n)}{n^2 \phi^2(i/n)}  + O\left(  \frac{1}{\sqrt{n i^3} } \right).
\end{equation}
And if we instead consider $ \frac{\partial H}{\partial y_i}$ we arrive at formulas similar to \eqref{DerivFormula2} and 
\eqref{logderiv} where: in the analog of the former, $B_m$ and $C_m$ change roles and $D_m$ changes sign, while in the latter $\oney$ 
replaces $\onex$ on the right hand side and $a$ is set to zero. The claim then is that the formulas for $\onex$ and $\oney$ are equivalent to:
\begin{align} 
%s=&\sum_{m=1}^d g_m m\binom{2m}{m} \phi(s)^{2m}  \\
\vect{a-\beta^{-1}-\frac{t}{\phi} \onex}{-\beta^{-1}-\frac{t}{\phi}\oney}
=& \sum_{m=1}^d \frac{g_m m}{2m-1}\binom{2m}{m}\phi^{2m-1} \label{x1ydef}   \\
& \hspace{.05\textwidth}\times \left(\begin{matrix} 
2m^2-2m+1 & 2m^2-2m & -m^2+m \\
2m^2-2m & 2m^2-2m+1 & m^2-m 
\end{matrix}\right) \left( \begin{matrix} \onex \\ \oney \\ \phi' \end{matrix}\right), \nonumber 
\end{align}
in which $t = i/n$, the $A_m$ terms having dropped out by the definition  \eqref{phi_def} of $\phi$. Now adding the rows
of \eqref{x1ydef} we find that the expression for $\onex(t) + \oney(t)$ from \eqref{finemin2} would follow from the identity
$$
   \sum_{m=1}^{d}  \frac{g_m m (4m^2 - 4m +1)}{2m-1} {2m \choose m} \phi(t)^{2m-1} = \frac{1}{\phi'(t)} - \frac{t}{\phi(t)},
$$
but  this is another consequence of \eqref{phi_def}. The expression for $\onex(t) - \oney(t)$ is checked in a similar way.

We remark that while this exact vanishing cannot hold for indices $i \in [1, d]$ or $[n-d, n]$, for $i = O(1)$ Lemma \ref{phi_bounds} implies that the right hand sides of both \eqref{DerivFormula2} and \eqref{logderiv} are $O(1/\sqrt{n})$ 
allowing the estimate \eqref{eq:GradH} to be extended to the lower range.

We now go back and derive the coefficients \eqref{coefficients}. In counting weighted $P_{m, r}$ paths 
it is convenient to introduce the following bijection. Denote by $\tilde{P}_{m,r}$ those 
paths in $P_{m, r}$ for which the first step is of type $\rightarrow$, and define
\begin{equation}
\label{bijection}
  (p,j) \in P_{m,r} \times \{ 1, \dots, r \} \mapsto (\tilde{p}, q) \in \tilde{P}_{m,r} \times \{ 0, \dots, 2m-1 \},
\end{equation}
where $\tilde{p}$ is obtained by shifting $p$ to the left so the the $j^{th}$ height-zero step of type $\rightarrow$ of $p$ becomes the first step of $\tilde{p}$. The number $q$ tracks how far $p$ is shifted to produce $\tilde{p}$.

For $A_m$ we select the same factor (the $\phi(i/n)$ from \eqref{shift_fine}) in each factor in the inner product of \eqref{DerivFormula2}, and so
\begin{equation}
\label{theA}
  A_m = \sum_{r=1}^{2m} r | P_{m,r} | = 2m | \tilde{P}_{m} | = 2 m \sum_{\ell = 0}^{m} {m-1 \choose \ell} { m \choose \ell}.
\end{equation}
The second equality uses \eqref{bijection} with $\tilde{P}_{m}$ denoting the union over all $  \tilde{P}_{m,r}$, that is,  
those paths just  constrained to start with a $\rightarrow$ step. Then we sum over choices of positions for the
 $\ell$ steps of type $\searrow$ among the remaining $m-1$ odd-timed steps, balanced by a choice of $\ell$ (of $m$ possible) steps of type $\nearrow$ at the even-timed steps. The last expression in \eqref{theA} can be written 
 $2 \sum_{\ell=0}^{m} \ell  {2m \choose \ell}^2$ after a change of variable which in turn equals $m {2m \choose m}$.

In computing $B_m$, one of the previous $\phi$ factors is now a $\onex$.  These can only appear at $\rightarrow$ steps,
and so in the sum over $P_{m,r}$ paths one has the weight  $(\# \{ \rightarrow \mbox{steps} \} - 1)$ to account for the possible choices of position of the $\onex$ factor. (The $-1$ shift is due the fact we have differentiated in the $x_i$ variable.)
Hence, with an obvious shorthand and by the same reasoning behind \eqref{theA}:
$$
%\begin{equation}
\label{theB}
  B_m = \sum_{r=1}^{2m} r \sum_{p \in P_{m,r}} ( \#_{\rightarrow}- 1) = 2 m \sum_{p \in \tilde{P}_m} ( \#_{\rightarrow} - 1)
    = 2m \sum_{\ell=0}^m  (2m-2 \ell -1)  {m-1 \choose \ell} { m \choose \ell}.
%\end{equation}
$$
Similar to before we can rewrite the above as $ \sum_{\ell=0}^m (2 \ell) (2 \ell -1)  {2m \choose \ell}^2$, from which the expression in \eqref{coefficients} follows from the derivation of $A_m$ along with known expressions for $\sum_{\ell=0}^m 
\ell^2  {2m \choose \ell}^2$. The calculation for $C_m$ is basically the same, with $\#_{\nearrow \cup \searrow}$ in place
of $\#_{\rightarrow} -1$.

Finally turning to $D_m$, note that any appearance of $\phi'$ is weighted by the relative height of the path, and we have that 
\begin{align}
D_m=&\sum_{r\geq 0} \sum_{p\in P_{m,r}} r \sum_{j=1}^{m} \left( p(2j-1)+\begin{cases} p(2j) & \text{if step $2j$ is $\rightarrow$} \\ p(2j)-1 &  \text{if step $2j$ is $\nearrow$} \end{cases} \right)  \nonumber  \\
=&\sum_{r\geq 0} \sum_{p\in P_{m,r}} r \sum_{j=1}^{m} 2p(2j-1) 
=  4m \sum_{p\in \widetilde{P}_m}  \sum_{j=1}^m p(2j-1),  \label{DFormula1}
\end{align}
where we have used that the corresponding weight is always the smaller of the heights across any step.
To evaluate the last sum in (\ref{DFormula1}) we use a method which we learned from \cite{EMP} (see in particular Prop.~4.2):
% and the second author extended in \cite{Waters15}.
\begin{align*}
\ \sum_{p\in \widetilde{P}_m}  \sum_{j=1}^m p(2j-1) 
=& \sum_{\substack{i_1,k_1,m_1,i_2,k_2,m_2\geq 0 \\ m_1+m_2 =m \\ i_1-k_1 +i_2 -k_2 =0}}
(i_1 -k_1) \binom{m_1-1}{i_1}\binom{m_1 -1}{k_1} \binom{m_2+1}{i_2} \binom{m_2}{k_2}  \nonumber  \\
=&[w^0] [z^m] 
%|_{u,v \rightarrow w} 
 \left( \sum_{i_1,k_1,m_1,i_2,k_2,m_2\geq 0} u \partial_{u} u^{i_1-k_1} v^{i_2-k_2}z^{m_1+m_2} \right. \nonumber \\
&
\left. \hspace{.2\textwidth}\times \binom{m_1-1}{i_1}\binom{m_1 -1}{k_1} \binom{m_2+1}{i_2} \binom{m_2}{k_2}
 \right) \Big|_{u, v \mapsto w}  \nonumber \\
=&[w^0] [z^m]
% |_{u,v \rightarrow w} 
\left(u \partial_{u} \frac{z}{1-z(u+2+u^{-1})} \frac{1+v}{1-z(v+2+ v^{-1})} \right) \Big|_{u, v \mapsto w}.   
%=&[w^0] (w-1)(w+2+w^{-1})^{m-1} \nonumber \\
%=&-\frac{m^2-m}{4(2m-1)} \binom{2m}{m}.\nonumber 
\end{align*}
In line one, $i_1$ and $k_1$ count the running number of type $\nearrow$ and $\searrow$ steps, respectively. In line 
two, we have used the notation  $[x^p] f(x)$ for the $p^{th}$ coefficient of the Taylor expansion of the (analytic) function
$x \mapsto f(x)$. The remaining evaluations are straightforward.
\end{proof}

\subsection{Minimizers and boundary conditions}

Here we consider the Hamiltonian $H$ subject to certain boundary conditions. To be more precise, start by fixing  an interval $I = [i_0, i_1] \subset [1, n]$ and denote by $\partial I$ (the boundary of $I$)  the at most $d$ indices to the left/right of $I$.
That is, $\partial I =  ([i_0-d, i_i-1] \cup [i_1+1, i_1+d] ) \cap [1,n]$. View $H$ as a function of $(x,y) \in  I$ with those coordinates whose indices lie in $\partial I$ prescribed to equal some values $q$. By the assumptions on $V$, those $(x,y) \in I \cup \partial I$ decouple from the $(x,y) \in I$ given $q$. This restricted function is referred to as the ``conditional Hamiltonian" $H_q$ with boundary conditions $q$.

The goal is to quantify at what rate the minimizers of $H_q$ become independent of $q$ as one moves away from the boundary.

\begin{proposition} 
\label{thm:BCdecay}
For an interval $I \subset [1, n]$ consider the conditional Hamiltonian $H_q$, i.e.,  $H$ restricted to $I$ with the coordinates in $\partial I$ set equal to some values $q$. Assume $\| q\|_{\infty} \le c'$. Then, with $(x^q, y^q)$ the minimizer of $H_q$, it holds that
\begin{equation}
\label{eq:BCdecay}
     | x_i^q - x_i^o | + | y_i^q - y_i^o | \le c \,   \| q -  (x^o,y^o) \|_{\infty, \partial I} \, e^{- \mathrm{dist}(i, \partial I)/c},
\end{equation}
for any $i \in I$.
Here $c = c(V, \beta, a, c')$.  
 If $I  = [i_0, i_1] \subset [1, n-d]$ there is also the bound,
\begin{equation}
\label{eq:BCdecay1}
     | x_i^q - \xbolt_i | + | y_i^q - \ybolt_i | \le   c \, \max \left( \frac{1}{ \sqrt{ n i_0^3}} , \   \| q -  (\xbolt, \ybolt) \|_{\infty, \partial I} \, e^{- \mathrm{dist}(i, \partial I)/c} \right),
\end{equation}
for any $i \in I$.
\end{proposition}

Proposition  \ref{thm:BCdecay} is a deterministic version of decorrelation, and will play an important role in the blocking estimates behind the functional central limit theorem in Section \ref{sec:CLT}. More presently it used to turn the calculation of Lemma \ref{lem:GradHBound} into a fairly optimal estimate on the distance between the true and fine minimizers (Proposition \ref{prop:FineMin} in the next subsection).

The proof of Proposition \ref{thm:BCdecay} is based on the following two lemmas. The first, Lemma \ref{lem:MinBounds},
is a direct consequence of the uniform convexity criteria \eqref{ConvexGift1} - \eqref{ConvexGift2}. This is then bootstrapped 
to yield the proposition with the help of Lemma \ref{lem:Gronwall}, which is a kind of discrete Gronwall inequality. The proof of the latter is a simple inductive argument which is not reproduced here.

The program  is quite similar to that in Sections 6-7 of \cite{KRV}. However, the use of \eqref{ConvexGift1} - \eqref{ConvexGift2} streamlines things considerably, bypassing for example the a priori lower bounds on minimizers required in \cite{KRV}. 

\begin{lemma}
\label{lem:MinBounds}
%The minimizer $(x^o,y^o)$ of $H$ has sup-norm bounded independently of the dimension $n$.
 For any conditional minimizer $(x^q, y^q)$ of an $H_q$ defined on some $I \subset [1,n]$, 
\begin{equation}
\label{qminbound1}
   \| (x^q, y^q) - (x^o, y^o)  \|_{2, I}^2  \le \rho(q) \| q - (x^o, y^o)  \|_{2, \partial I}^2 
\end{equation}
where $\rho(q)$ is polynomial of degree $2d$ in the boundary variables $q$ (with bounded coefficients depending only on $V, \beta,$ and $a$).
And if $I \subset [1,n-d]$ it also holds that
\begin{equation}
\label{qminbound2}
   \| (x^q, y^q) - (\xbolt, \ybolt)  \|_{2, I}^2  \le   \sum_{i \in I} \frac{c}{{n i^3}}   + \rho(q) \| q - (\xbolt, \ybolt)  \|_{2, \partial I}^2 
\end{equation}
with another polynomial $\rho$ of degree $2d$ and $c = c(V, \beta, a)$.
\end{lemma}

\begin{lemma}
\label{lem:Gronwall}
Let $a_i$  and $b_i$ be nonnegative sequences satisfying 
$
   \sum_{i=0}^k a_i \le c a_{k+1} + \sum_{i=0}^k b_i 
   %\mbox{ for all } i  = 0, 1 , \dots, m.
$
for a constant $c$ and all $k \le m$. Then it holds that
$$
   a_0 \le c \left( \frac{c}{c+1} \right)^k a_{k+1} + \sum_{i=0}^k  \left( \frac{c}{c+1} \right)^i b_i,
$$
again for all $k \le m$.
\end{lemma}

\begin{proof}[Proof of Lemma \ref{lem:MinBounds}] 
We first observe that $(x^o, y^o)$ is bounded in sup-norm, independent of the dimension.
By  \eqref{ConvexGift2}: with a different constant $c$, 
$$
  \|  (\xbolt, \ybolt) - (x^o, y^o) \|_{\infty}^2 \le \| (\xbolt, \ybolt) - (x^o, y^o) \|_{2}^2 \le  c 
   \| \nabla H (\xbolt, \ybolt) \|_2^2,
 % \sum_{i=1}^n 
  % \left| \frac{\partial H}{\partial x_i} (
%     \left| \frac{\partial H}{ \partial x_i}(\xbolt, \ybolt) \right|^2 +  \left| \frac{\partial H}{\partial y_i} (\xbolt, \ybolt)  \right|^2,
$$
and this is $O(1)$ by Lemma \ref{lem:GradHBound}. The explicit formulas \eqref{finemin2} and \eqref{finemin1}  then show that $\| (\xbolt, \ybolt) \|_{\infty}$ is bounded independently of $n$ which yields the  claim.

Similarly, 
\begin{equation}
\label{L2Bound1}
   \|  (x^o, y^o) - (x^q, y^q) \|_{I,\infty}^2 \le 
   %\| (\xbolt, \ybolt) - (x^o, y^o) \|_{2}^2 \le  c \sum_{i=1}^n 
  % \left| \frac{\partial H}{\partial x_i} (
    c \sum_{i \in I: \, {\mathrm{dist}}(i, \partial I) \le d }  \,   \Bigl(  | \partial_{x_i} H_q (x^o, y^o) |^2 +  
    | \partial_{y_i} H_q  (x^o, y^o)  |^2 \Bigr),
\end{equation}
since if $i \in I$ with $\mathrm{dist}(i, \partial I) > d$, we have that 
$(\partial H_q /\partial z_i)(x^o,y^o) = (\partial  H/\partial z_i)(x^o,y^o) = 0 $ for $z_i = x_i$ or $y_i$. For the remaining $2d$ terms denote by $P_V$ the polynomial part of $H$ and note: with again $z_i = x_i$ or $y_i$,
$$
  | \partial_{z_i} H_q (x^o, y^0 )|^2 = | \partial_{z_i} P_V(x^o, y^o) - \partial_{z_i} P_V(x^o, y^o; q) |^2,
 $$
 where the notation indicates that coordinates in $\partial I$ are evaluated at either the entries of $(x^o, y^o)$ or the corresponding $q$.  But by pairing entries the above is bounded by a sum of $|q -z_i|^2$ terms with coefficients that are polynomials (of degree at most $2d$) in $(x^o, y^o, q)$. As we have just shown that 
$\| (x^o, y^o)\|_{\infty} $ is uniformly bounded, all these polynomial factors can be further controlled above by a $\rho(q)$ with the claimed properties.

For \eqref{qminbound2} we basically repeat the argument. The key difference being that in the estimate corresponding to
\eqref{L2Bound1}, the sum over $i \in I : {\mathrm{dist}}(i, \partial I) > d $ on the right hand side does not vanish, but instead produces a multiple of $\sum_{i\in I} \frac{1}{n i^3}$, courtesy Lemma \ref{lem:GradHBound}. This also explains the restriction of $I$ in this case to $[1, n-d]$.
\end{proof}

\begin{proof}[Proof of Proposition \ref{thm:BCdecay}]
Consider first \eqref{eq:BCdecay1}. The idea is to apply the inequality  \eqref{qminbound2} of Lemma \ref{lem:MinBounds}
to a well-chosen collection of subintervals of $I$. 

Fix an index $i \in I$ and  decompose $I \cup \partial I$ into consecutive blocks $I_{-m}, \dots, I_m$ with $\partial I  = I_{-m} \cup I_{m}$ such that $i \in I_0$ and each $I_j$ for $j \neq 0$ is of length $d$.  Denote $J_k = \cup_{|j| \le k } I_j$ and $\partial J_k = I_{-k-1} \cup I_{k+1}$. Then,  as 
a consequence of \eqref{qminbound2}, we have that
\begin{equation}
\label{thesecret}
   \| (x^q, y^q) -  (\xbolt, \ybolt)  \|_{2, J_k}^2  
                                      \le  c \| (x^q, y^q) -  (\xbolt, \ybolt)    \|_{2, \partial J_k}^2 +  c \sum_{i \in J_k} \frac{1}{n i^3}, 
\end{equation}
for $k = 0,\dots m-1$. 

Two remarks are in order.
First, a direct application of  \eqref{qminbound2} would have the constant $c$ multiplying 
$\| (x^q, y^q) -  (\xbolt, \ybolt)    \|_{2, \partial J_k}^2$ by a polynomial in the variable $(x^q, y^q)$ appearing in $J_k$. But \eqref{qminbound2} (or \eqref{qminbound1}) also shows that every $|x^q|$ and $|y^q|$ is bounded by the same polynomial in the variables $q$. By assumption the $q$ are bounded, and so it is possible to use the same constant $c$ (for all $k$) throughout \eqref{thesecret}.
Second, in case the definition of any $I_j$ places it outside of $\{ 1, n-d \}$  the corresponding sum  is simply taken as empty. This allows the conclusion to extend to one-sided minimizers; when for example $I$ is of the form $[1,L]$ with boundary conditions placed at $[L+1, L+d]$. 
 
Returning to \eqref{thesecret}, this system is exactly as in the hypothesis of Lemma \ref{lem:Gronwall} with 
 $$
 a_j = \| (x^q, y^q) -  (\xbolt, \ybolt)  \|_{2, I_{-j} \cup I_j}^2, \quad b_j =   \sum_{i \in I_{-j} \cup  I_j}  \frac{c}{n i^3},
 $$
 and so there is a constant $c'$ for which
 $$
   a_0 \le c' e^{-m/c'} a_m + c' \max_{j \in [0, m-1]} b_j.
 $$
 This is recognized as \eqref{eq:BCdecay1} upon noting: $a_0 \ge ( | x_i^q - \xbolt_i | + | y_i^q - \ybolt_i | )^2$, 
 $a_m \le d \| q -  (\xbolt, \ybolt) \|_{\infty, \partial I}^2$, and $b_j \le \frac{1}{n i_0^3}$ for al $j \in [0, m-1]$.
 The proof of  \eqref{eq:BCdecay} is identical save that in that case $b_j \equiv 0$.
\end{proof}

\subsection{The limiting mean}

The results of the previous subsection yield the following.

\begin{proposition}
\label{prop:FineMin}
There is a constant $c = c(V, \beta, a)$ so that
\begin{equation}
\label{eq:fine_est}
   |  (\xbolt_i, \ybolt_i) - (x_i^o, y_i^o) |  \le  c'  \left\{   
   \begin{array}{ll} \frac{1}{\sqrt{n}} & \mbox{ for } i \le c \log n, \\
      \frac{1 }{\sqrt{n i^3}}  & \mbox{ for } c \log n  < i \le n - c \log n, \\
                    e^{-(n-i)/c'} & \mbox{ for }   n - c \log n < i \le n, \end{array}\right. 
\end{equation}
with a constant $c'$ depending on $c$.
It follows that
\begin{align}
\label{eq:mean1}
   \sum_{k= \lfloor n s \rfloor}^{\lfloor nt \rfloor} \log \frac{x_k^o}{y_k^o} 
     =  &   - \left( \frac{a}{2} + \frac{1}{4} \right) \log \frac{\theta(t)}{\theta(s)} +  \frac{1}{2} \log \frac{\phi(t)}{\phi(s)} \\
        & \hspace{.5cm} +  
 O \left( \frac{1}{ns}  + \sum_{k= \lfloor ns \rfloor}^{ \lfloor c \log n\rfloor} \frac{1}{\sqrt{k} } 
 + \sum_{k= \lfloor n (1-t) \rfloor}^{ \lfloor c \log n \rfloor}  e^{-k/c}   \right), \nonumber
\end{align} 
for all $0 \le s < t \le 1$.
\end{proposition}

For fixed $ s < t $ bounded away from $0$ and $1$, the error term on the righthand side of \eqref{eq:mean1} reduces to $O(1/n)$ and we have  advertised  limiting mean in line one, recall \eqref{mean_convergence}.  The estimate is stated in this more complete form for use in Section  \ref{sec:Norm}.

\begin{proof} For \eqref{eq:fine_est} start with the case that $i$ is a  distance $ O( \log n)$ away from both $1$ and $n$. Let $I = [i - c \log n, i + \log n] $ and consider the conditional Hamiltonian $H_q$ with $q = (x^o, y^o)$ on $\partial I$.  The conditional minimizer is then the true minimizer through $I$ and we can apply \eqref{eq:BCdecay1} of Proposition \ref{thm:BCdecay}
with $(x^q,y^q) = (x^o, y^o)$. The result is that
\begin{equation}
\label{eq:BCcase1}
     |  (\xbolt_i, \ybolt_i) - (x_i^o, y_i^o) |   \le c' \max  \left( \frac{1}{\sqrt{n (i - c \log n)^3 }}, e^{ - c'' \log n} \right), 
\end{equation}
having used that $(x^o, y^o)$ has bounded entries (proved in the coarse of establishing Lemma \ref{lem:MinBounds}). But $c''$ can be made large with $c$ and by choice $ i - c \log n = O(i)$.
The other two cases are similar.  For example for $i \le c \log n$, consider $I = [1, c'' \log n]$ with one-sided boundary conditions $= (x^o, y^o)$ on the $d$-length stretch to the right of $c'' \log n$ (and $c'' \gg c$). Then the boundary component in the analog of \eqref{eq:BCcase1} can be made smaller than any inverse power of $n$, but the $O(n^{-1/2})$ stemming from $i_0 = 1$ cannot be beat.

 Moving to \eqref{eq:mean1} we start by noting that
\begin{equation*}
  \log  \frac{\xbolt_k}{\ybolt_k}  
     = \log \left( \frac{1 + \frac{\onex (k/n)}{n \phi(k/n)}}{1 + \frac{\oney (k/n)}{n \phi(k/n)}} \right)
    = 
   \frac{ \onex(k/n) - \oney(k/n)}{n \phi(k/n)} + O(k^{-2}), 
 \end{equation*}
 by the estimates of Lemma \ref{phi_bounds}. Summed over $[ns, nt]$ this contributes to $O( (ns)^{-1} )$ to the advertised error.  
 Next we have that,
\begin{equation*}
  \sum_{k= \lfloor n s \rfloor}^{\lfloor nt \rfloor}  \frac{ \onex(k/n) - \oney(k/n)}{n \phi(k/n)} = \int_s^t   \frac{ \onex(u) - \oney(u)}{\phi(u)} du + O \left(\frac{1}{ns} \right),
\end{equation*} 
where the integral equals the right hand side of the first line in \eqref{eq:mean1}. The error here
follows from the standard Riemann sum bound  given that 
$ \left| \frac{d}{du} \left(  \frac{ \onex(u) - \oney(u)}{\phi(u)} \right)  \right| = O (u^{-2})$  for small $u >0$, again by Lemma
\ref{phi_bounds}. The remaining overall error term is (the sum of)
\begin{equation*}
    \log \left( 1 + \frac{x_k^o - \xbolt_k}{\xbolt_k} \right) -   \log \left( 1 + \frac{y_k^o - \ybolt_k}{\ybolt_k} \right)
     =  \left\{   
   \begin{array}{ll}  O ( k^{-1/2} )& \mbox{ for } k \le c \log n, \\
       O(k^{-2})  & \mbox{ for } c \log n  < k \le n - c \log n, \\
                   O( e^{-(n-k)/c} ) & \mbox{ for }   n - c \log n < i \le n, \end{array}\right. 
\end{equation*}
Here we have used  \eqref{eq:fine_est} and the fact that $\xbolt_k, \ybolt_k =O(\sqrt{k/n})$ for small $k$.
The first and third bounds on the right hand side explain the final terms in line two of \eqref{eq:mean1}.
\end{proof}

\section{Gaussian concentration and approximation}
\label{sec:Gauss}

We build up yet more technical machinery. First we establish a sharp form of 
Gaussian concentration for $P$ about the minimizer $(x^o, y^o)$ of the Hamiltonian $H$.
Along the way we see that similar concentration holds for the conditional distributions of $P$, or those measures of the form 
$P_q$ with density  proportional $e^{- n \beta H_q(x,y)}$ restricted to the corresponding interval $I$. Here $H_q$ is the conditional Hamiltonian 
with boundary conditions $q$ on $\partial I$ introduced in the last section. These estimates are then used to establish approximations of certain $P$ expectations by their Gaussian counterparts.

\subsection{Concentration}
\label{sec:Conc}

We re-emphasize that we are assuming $\beta \ge 1$. Our main Gaussian concentration result is the following.

\begin{proposition}
\label{prop:Concentration} 
There is a constant $c'$ depending on $V, \beta,$ and $a$  such that for $t > \frac{c'}{\sqrt{n}}$ we have 
\begin{equation}
\label{eq:Concentration}
   P \left( | X_k - x_k^o | + |Y_k - y_k^o | > t \right) \le c e^{-n t^2/c}
\end{equation}
for any $k \in [1, n]$ and $c$ depending on $c'$. 
\end{proposition}

This serves as a refinement of the Brascamp-Lieb type inequality proved as Lemma 8.1 of \cite{RRV}:

\begin{lemma} 
\label{prop:oldconc}
There is a constant $c = c(V, \beta,a)$ so that $\| (X,Y)- (x^o, y^o) \|_2$ is stochastically dominated by $\|G \|_2$ where $G$ is the Gaussian vector on $\RR^{2n-1}$ with density proportional to $e^{-c n \| g \|_2^2/2}$.  Additionally, under any $P_q$  we have that
$\| (X,Y)- (x^q, y^q) \|_{2,I}$ is similarly dominated by the norm of a Gaussian vector in dimension $2 |I|$ with entry variance
$(cn)^{-1}$.
%Consider a measure $\mu$ supported on a convex subset $A$ of $\RR^n$ with density  proportional to $e^{-H(z)} = e^{-H(z_1, \dots, z_n)}$ for a uniformly convex $H: A \mapsto \RR$.  Assume further that the minimum $z^o$ of $H$ is achieved on $A$ 
%and denote by $c$ the convexity constant of $H$.
% $\mbox{Hess} \,  H(z) \ge c I$. 
%Then $\| Z- z^o \|_2$ 
%is stochastically dominated by $\|G \|_2$, where $G$ is the Gaussian vector on $\RR^n$ with density proportional to $e^{-c \| g \|_2^2/2}$.
\end{lemma}

Note that \cite{RRV} states the above in a far more general way. What is important for the present application is that the convexity constant for any conditional $H_q$ can be bounded below by that of $H$.
Lemma \ref{prop:oldconc} may be iterated to produce increasingly better tail estimates on local scales (or shorter stretches of indices). While this sufficed for the soft edge problem in \cite{RRV}, for the control required below it is more efficient to take a different approach. 

Here we rely on the fact that, with $\beta \ge 1$,  $P$ satisfies a Logarithmic Sobolev Inequality (see for example \cite{BL}). Given that,  ``Herbst's argument" (see Theorem 5.3 of \cite{L}) yields:

\begin{lemma}
\label{lem:LSI}
There is a constant $c = c(V, \beta, a)$ such that
$$
   P ( | F(X,Y) - E F(X,Y)  | > t) \le 2 e^{- c t^2/2 \| F \|_{Lip}^2},
$$
for any Lipschitz  function $F: \RR_{+}^{2n-1} \mapsto \RR$.
\end{lemma}

Applied to  $F(x,y) = x_k$ or $y_k$ this produces an inequality of the form \eqref{eq:Concentration} for all $t >0$, though centered at the mean rather than at the minimizer.   Proposition \ref{prop:Concentration} then follows from the next estimate, which actually makes essential use of the old Gaussian concentration result Lemma \ref{prop:oldconc}.

\begin{lemma}
\label{lem:mean}
It holds that
$$
   E \Bigl[ | X_k - x_k^o|  +  | Y_k - y_k^o | \Bigr] \le \frac{c}{\sqrt{n}}, 
$$
for any $k \in [1,n]$ and $c = c(V, \beta, a)$.
\end{lemma}

\begin{proof}
Throughout we use the shorthand $Z$ (or $z$) to denote the pair of variables $(X,Y)$ (or $(x,y)$). 

To start we fix an interval $D = \{ k \in [ \ell - d/2, \ell+d/2] \} $, and for a choice of $m$ (to be determined) let $I$ be the interval made up of $D$ and the (at most) $m$ indices to the left/right. As usual $\partial I$ will denote the $d$ indices to the left and/or right of $I$. If any part of $D$, $I$, or $\partial I$ falls outside of $[1, n]$, it is truncated in the obvious way or viewed as empty.  Then, with $E_q$ the conditional expectation with respect to the variables $q \in \partial I$ we write
\begin{align}
\label{iterate}
  E \max_{k \in D} |Z_k -  z_k^o| & \le  E E_q \max_{k \in D} |Z_k  -  z_k^q| + E E_q   \max_{k \in D} |z_k^o  -  z_k^q|. 
\end{align}
Here  $z^q$  denotes the conditional minimizer of $H_q$ on $I$ (with boundary conditions $q$). By passing the randomness onto the variables $q$ (for which we continue to use lower case) we will be able to iterate this inequality.

Further bounding \eqref{iterate} above we have that
\begin{equation}
\label{L2bound}
    E E_q \max_{k \in D} |Z_k  -  z_k^q|  \le E E_q \| Z - z^q \|_{2, I}  \le c \sqrt{\frac{m}{n}}
\end{equation}
by  Lemma \ref{prop:oldconc}. 
And by Proposition \ref{thm:BCdecay} we also have the bound: with  $Q_{\partial I}$ the event that $| q - z^0|_{\infty, \partial I} < b$,
 \begin{equation}
\label{BCdecay}
    \max_{k \in D} | z_k^o  -  z_k^q|  \one_{Q_{\partial I}}  \le  c' e^{-m/c}  \max_{i \in \partial I} | z_i^o - q_i |,
\end{equation}
for a $c = c'(b)$. And using Lemma  \ref{lem:MinBounds} on the complement of $Q_{\partial I}$: 
\begin{align*}
  E E_q   \max_{k \in D} |z_k^o  -  z_k^q|  
  %&& \le  e^{-m/c} E   \left[  P(q, z^0)   \max_{i \in \partial I} | z_i^0 - q|  \right] \\
    & \le  c'  e^{-m/c}  E  \Bigl[ \max_{i \in \partial I} | z_i^o - q_i |  \Bigr] +
     E  \Bigl[ \rho(q)   \| z^o - q \|_{2, \partial I}  , Q_{\partial I}^c \Bigr]. 
\end{align*}
By Holder's inequality and another application of Proposition \ref{prop:oldconc}  we  can control the second term on the right hand side  by a constant multiple of $P ( Q_{\partial I}^c ) \le c' e^{- n/c'}$ with a new $c' = c'(b)$.

Adjusting constants and substituting this last estimate along with \eqref{L2bound} into \eqref{iterate} gives
\begin{equation}
\label{almostfinal}
  E \max_{k \in D} |Z_k - z_k^o|  \le c  \sqrt{\frac{m}{n}} + c e^{-m/c}  E  \left[ \max_{i \in \partial I} | z_i^o - q_i |  \right] + c e^{-n/c}.
\end{equation}
At this point we can choose $m$ large enough (but independent of $n$) so that $c e^{-m/c} < 1/4$ and then for large enough $n$ absorb the final term on the right hand side into the first. Then \eqref{almostfinal} may be schematized as in:
\begin{equation}
\label{finalexpbound}
  a_k \le 2 c  \sqrt{\frac{m}{n}}  + \frac{1}{4} (a_{k-m} + a_{k+m}).
\end{equation}
Here $a_k$ is $  E \max_{i \in D} |Z_i-z_i^o| $ for whatever interval $D$ centered at $k$ with $a_{\ell} = 0 $ for $\ell \le 0$ or $\ell \ge n$.  In this interpretation, the $q $ in the expectation on the right hand side of  \eqref{almostfinal} stands in for the corresponding (random) $Z$ variable while $\partial I$ serves as  a shifted copy of $D$.

The claim is that \eqref{finalexpbound} (along with its corresponding side conditions) implies 
 all the $a_k = a_k(n)$ are bounded by a constant multiple of $n^{-1/2}$.
After a scaling the problem can be summarized thus:  Given an array $b_k = b_k(L)$, for $k = 0,\dots, L$ which is nonnegative, finite and satisfies (for each $L$),
\begin{equation}
\label{beq}
   b_k \le 1 + \frac{1}{4} (b_{k-1} + b_{k+1}), \quad  b_0  =  b_L = 0, 
\end{equation}
there is a constant which bounds all $b_k $ independently of $L$.  This can be seen by contradiction: if for some $j$ we have  that say $b_j \ge 4$,
any such solution 
 must grow exponentially to either the left or right of $j$. But this would violate the  Dirichlet boundary conditions imposed at $j=0$ or $j=L$. 
\end{proof}

\subsection{Laplace estimates}
\label{sec:Laplace}

Along with  $P_q(dx,dy) = \frac{1}{Z_q} e^{-n \beta H_q(x,y)} dx dy$ introduce the natural Gaussian measure approximating $P_q$ over the same interval $I$:
\begin{equation}
\label{TargetG}
   \nu_{I, q}(dx, dy) = \frac{1}{{Z}_q^{\prime}} \exp{ \left( - \frac{n \beta}{2}  
   \Big\langle (x - x^q, y-y^q ),  \mathcal{H}_q (x - x^q, y-y^q ) \Big\rangle \right)} dx dy,
\end{equation}
where $\mathcal{H}_q$ denotes the Hessian of $H_q$ evaluated at $(x^q,y^q)$. We also bring in the mixture of $\nu_q$ over boundary conditions in ``typical" position, defined by
\begin{equation}
\label{TargetMix}
   \int F(x,y)  \mu_{I, c}(dx, dy) = E \Bigl[ \int F(x,y) \nu_{I,q}(dx ,dy),  \, \| q -(x^o, y^o) \|_{\infty, \partial I} \le c \delta_n
    \Bigr]
\end{equation}
with $\delta_n = \sqrt{\frac{ \log n}{n}}$.

To determine the statistics of the field for bulk indices, we have the following estimate  which relates the $P$-expectation of certain polynomial test functions to those of averaged Gaussians.

\begin{proposition} 
\label{prop:testfunctions}
Fix a small $\delta > 0$ and let $I \subset [ \delta n , n]$. Denote by $K$ the interval  made up of $I$ along with the (at most) $ c \log n$ indices to its left and right. Let $F_I(x', y')$ be a nonnegative polynomial with bounded coefficients and of bounded degree in the variables $(x_i - x_i^{\prime})$ and $(y_i - y_i^{\prime})$ for $i \in I$ and prescribed centerings $(x', y')$.   Then, there exist constants $ c$ and $ c'$  (which depend on $V, \beta, a$ and the degree of $F_I$) 
such that
\begin{equation}
\label{eq:testfunctions}
 E \Bigl[  F_I (x^o, y^o) \Bigr] =  \left( \int   F_I (x^q, y^q)  \mu_{K,c'}(dx, dy) + O(n^{-2}) \right) 
 \left(   1+ O \left( \frac{|K| (\log n)^{3/2}}{\sqrt{ n}} \right) \right).
\end{equation}
Of course, $c$ and $c'$ figure into the implied constants in the error terms and the estimate \eqref{eq:testfunctions} presumes that $|K| (\log n)^{3/2} n^{-1/2} = o(1)$.
\end{proposition}

As a consequence of Proposition \ref{prop:testfunctions} the limiting variance of the field
$k \mapsto (X_k,Y_k)$ will be determined through:

\begin{corollary}
\label{lem:VarEst}
 Now let $I = [i_o, i_1]$ of length at most $n^{1/4-}$ and  supported in $[ \delta n, n -  2c \log n]$ for fixed $\delta > 0$ and large $c= c(V, \beta, a)$. With $K = [i_o - c \log n , i_1 + c \log n]$, denote by $P_q$ be the conditional measure on $K$ with boundary conditions $q$ satisfying $|| q - (x^o, y^o)||_{\partial K, \infty} \le c' 
 \sqrt{ \frac{\log n}{n}}$ (for a large $c' = c'(V, \beta,a)$). Then
 \begin{equation}
 \label{VarCalc1}
    E_q \left[   \left( \sum_{i \in I} (X_i - x_i^o) - (Y_i - y_i^o) \right)^2 \right] =   \phi^2(i_0/n) \frac{\theta'(i_0/n)}{\theta(i_0/n)}   
      \frac{(i_1-i_0)}{\beta n} + 
    O \left( \frac{(\log n)^2}{n} \right),
 \end{equation} 
 where the implied constant in the error term depends only on $\delta,  V, \beta , a, c, c'$. In particular, the same estimate holds with $E$ in place of $E_q$.
\end{corollary}

Last,  we will require the more particular control for indices down to $O(\log n)$ away from the singularity.

\begin{corollary} 
\label{prop:Shifts}
It holds that
\begin{equation}
\label{eq:MeanShift}
     E \left[ (X_k - x_k^o) - (Y_k - y_k^o) \right] = O \left(   \frac{(\log k)^{5}}{\sqrt{n}  k} \right),
\end{equation}
and
\begin{equation}
\label{eq:VarShift}
    E \left[ (X_k - x_k^o)^2 - (Y_k - y_k^o)^2 \right] = O \left(  \frac{ (\log k)^{5/2}}{n \sqrt{ k}}  \right),
\end{equation}
uniformly for $k \in [c \log n, n- c \log n]$ with $c = c(V, \beta, a)$ sufficiently large.
\end{corollary}

While Corollary \ref{lem:VarEst} is a direct calculation based on  Proposition \ref{prop:testfunctions}, the proof of 
Corollary \ref{prop:Shifts} entails  that  a higher order expansion be made than that behind the estimate \eqref{eq:testfunctions}. (The unattractive log factors
in \eqref{eq:MeanShift} and \eqref{eq:VarShift} could be improved by yet a higher order expansion, but the above suffices for what we will need.)

\begin{proof}[Proof of Proposition \ref{prop:testfunctions}]  To be concrete,  we will assume that
there are constants $p, q$ so that $ F_I  \le q + q \|(x,y) - (x^o, y^o) \|_{I, 2}^p$. It will be clear in the course of the argument that other choices of (bounded coefficient and bounded degree) polynomial $F_I$ 
will only alter the choices of $c'$ and $c$ made along the way.

With $F_I$ as specified and $Q_{\partial K} = \{  ||q - (x^o, y^o)||_{\partial K, \infty} \le {c'}{\sqrt{ \frac{\log n}{n}}} \}$, we first claim that by choice of $c'$ and $c$:
\begin{align}
\label{OpenVar}
  E \Bigl[ F_I( x^o, y^o)  \Bigr] & = E \Bigl[ E_q [ F_I( x^o, y^o) ] , Q_{\partial K} \Bigr] + O(n^{-2})  \\
   &  = E  \Bigl[ E_q [ F_I( x^q, y^q) ] , Q_{\partial K} \Bigr] + O(n^{-2}),  \nonumber
\end{align}
where  $(x^q, y^q)$ refers to the minimizer of the corresponding $H_q$ on the larger interval $K$. (Note: if the right edge of support of $I$ is less than $c \log n$ away from $n$, we are considering a one-sided minimization with boundary conditions
placed to the left of $K$.)
 For line one, Lemma \ref{prop:oldconc} shows $E F_I^2 = O(1)$  while for $c'$ large enough $P(Q_{\partial K}^c) = O(n^{-4})$ by 
Proposition \ref{prop:Concentration} $-$ indeed the $(-4)$ may be replaced by any negative power by choice of $c'$. Then apply Cauchy-Schwartz. For line two we assume that $c$ is large enough depending on $c'$ , so that Proposition \ref{thm:BCdecay} provides: for any 
$i \in I$,  $ | x_i^q - x_i^o|, |y_i^q-y_i^o|
\le c'' n^{-4} $ with a $c''$  uniform over $q \in Q_{\partial K}$. A second application of Cauchy-Schwartz using Lemma \ref{prop:oldconc} to control the $E_q$ expectation of powers of $F_I(x^q, y^q)$ produces the estimate.

At the expense of another $O(n^{-2})$ error, we can now further restrict the inner $E_q$ expectation in \eqref{OpenVar} to the event
 $Q_K$ on which  $ | X_i - x_i^q|, | Y_i - y_i^q| \le c''' \sqrt{ \frac{\log n}{n}}$ for all $i \in K$ for some $c'''$. This is a repetition of the argument 
employed in the first estimate of \eqref{OpenVar} coupled with the fact that  Proposition \ref{thm:BCdecay} gives that
$| x_i^q - x_i^o| \vee | y_i^q - y_i^o|$ is $O(\sqrt{\log n/ n})$ throughout $K$  (with the sharper estimate used just above holding on $I$)

Now we are in position to approximate $P_q$
by the Gaussian measure $\nu_q = \nu_{K, q}$ introduced in \eqref{TargetG}. If necessary we can adjust $Z_q$ (the $P_q$ normalizer) so that $H_q(x^q, y^q) = 0$. Then, with
\begin{align}
\label{3rdOrder}
    H_q^{(3)}(x,y)  =    H_q(x,y) - \Big\langle (x - x^q, y-y^q ), \frac{1}{2} \mathcal{H}_q (x - x^q, y-y^q ) \Big\rangle, 
       % &  \le \gamma  \left( \frac{\log n}{ n} \right)^{3/2},
\end{align}
Taylor's formula gives that
\begin{equation}
\label{3rdOrderB}
  | H_q^{(3)}(x,y) | \le  \sup_{t \in [0,1]} \sum_{k \in K}          
     \left(    \rho_k(t) + \frac{4k}{n |x_k(t)|^3} 
           + \frac{4k}{n |y_k(t)|^3}  \right) | (x_k, y_k) - (x_k^q, y_k^q)  |^3.  
\end{equation}
Here $\rho_k = \rho_{V,k}$ indicates a (positive) polynomial of fixed degree in the variables $(x_i(t), y_i(t))$ for ${i \in [k-d, k+d]}$, while $(x(t), y(t))$ draws out the line between $(x^q, y^q)$ and $(x,y)$. We use the fact that $\tr V (B B^T)$ is finite-range, so there are fixed number of mixed third-partial derivative involving any index $k \in K$ stemming from the polynomial of $H_q$. For the factors  corresponding to the third derivatives of the log terms in $H_q$, note that any $k\in K$ under consideration is large enough so that
$4k \ge 2(k+|a| +1/\beta)$. 

Further, with the left endpoint of $K$ at least $\delta n /2$, 
the results of the Section \ref{sec:Min} gives that: restricted to $Q_K$,
 $(x^q, y^q)$, and so also
$(x(t), y(t))$,   are bounded above and below independently of $n$ or $q \in Q_{\partial K}$. Therefore,
\begin{equation}
\label{Cheap3rdOrder}
   n   \beta | H_q^{(3)} (x, y) | \one_{Q_K} \le \gamma |K|   \left( \frac{ (\log n)^{3/2}}{n^{1/2}} \right)
\end{equation}
with $\gamma$ depending only on the parameters $\delta, c, V, \beta, a$. 
This allows the conclusion that
\begin{equation}
\label{GaussInt1}
  E_q [ F_I \one_{Q_K} ] =  \int_{Q_k}  F_I( x^q, y^q) \nu_{q}(dx, dy) 
                                                 \left( 1 + O \left( \frac{ (\log n)^{3/2} |K| }{\sqrt{n}}   \right)\right), 
\end{equation}
which is effectively the claim. Here we have used that, with $\eta$ the right hand side of 
\eqref{Cheap3rdOrder},
\begin{equation}
\label{Zsandwhich}
   e^{- \eta} \nu_q (Q_K) \le \frac{Z_q}{Z_q'} \le e^{ \eta} \frac{1}{P_q(Q_k)}.
\end{equation}
The Logarithmic Sobolev Inequality for Gaussian measures gives that $\nu_q(Q_K)$ is  the same order 
as $P_q(Q_K)$ (the measures were built to have the same convexity constant). That is, with each of these factors a negative power of $n$, the upper and lower bounds in \eqref{Zsandwhich} are controlled by $ e^{\pm  \eta} =  1 + O( \eta)$ with $\eta = o(1)$.
\end{proof}

\begin{proof}[Proof of Corollary \ref{lem:VarEst}]  Denoting by 
$S_I( x^o, y^o)$ the squared sum within the expectation of \eqref{VarCalc1},  the (beginning of) the proof of Proposition \ref{prop:testfunctions}  yields
$ E_q [ S_I( x^o, y^o) ]$ $ = E_q [ S_I( x^q, y^q)  \one_{Q_K} ]$ $+ O(n^{-2}).$
Again,  $Q_K $ is the event that $ \| (x,y) - (x^q, y^q) \|_{\infty, K}$ is less than $c' \sqrt{\frac{\log n}{n}}$ (for
choice of $c'$).
% is chosen to make  $(1 - P_q(Q_K)) E_q [S_I(x^q, y^q)^2]$ as small as necessary.
%As before, the choice of $c$  controls $\| (x^o, y^o) - (x^q, y^q) \|_{I, \infty}$. 
Continuing, the same proposition gives  that
\begin{equation}
\label{GaussInt11}
  E_q [ S_I (x^q, y^q) \one_{Q_K} ] =  \left( \int S_I(x^q, y^q) \nu_q(dx, dy)   + O(n^{-2}) \right)
                                                 \left( 1 + O \left( \frac{ (\log n)^{3/2}  |K|}{\sqrt{n}}   \right)\right),
\end{equation}
and  we will show that the advertised appraisal \eqref{VarCalc1} holds
for  the remaining $\nu_q$ integral.  This is enough since
the multiplicative error in  \eqref{GaussInt11}, can be combined with the $ O(  n^{-1} |K|)$ leading order term in  \eqref{VarCalc1} to be absorbed into a second 
$o((\log n)^2/n)$ additive error. Note that  
%with $i_0 > \delta n$,
the factor  $\phi^2(i_0/n) \frac{\theta'(i_0/n)}{\theta(i_0/n)}$ is of order one.

Next compute: 
%with $\mathcal{H}_q  = (\mathrm{Hess} \, {H_q})(x^q, y^q)$, 
\begin{equation}
\label{VarG}
   \int  S_I(x^q, y^q) \nu_q(dx, dy) =  \frac{1}{n \beta} w^T \mathcal{H}_q^{-1} w^{},
\end{equation}
where the right hand side is read as follows.  Indexing the integral (and so $\mathcal{H}_q$) according to
$(x_{k_0}, y_{k_0}, x_{{k_0}+1}, y_{{k_0}+1} \dots, x_{k_1}, y_{k_1})$ where $[k_0, k_1] = K$, the $(2|K|)$-vector $w$  has entries $(-1)^{i}$ for indices corresponding to the coordinates in $I \subset K$   and is otherwise zero. 
To estimate $  w^T \mathcal{H}_q^{-1} w$ we approximate $\mathcal{H}_q^{-1}$ by its ``coarse" version and find an exact eigenvalue problem. Here is where we will use the assumption that $I$ is supported $O( \log n)$ away from $n$. 

Recall the coarse Hamiltonian on  $K$ at position $t =  \frac{(k_0+k_1)}{2n} = \frac{(i_0+i_1)}{2n}$ introduced in \eqref{IntrocoarseH},
$$
  H_{K}(x,y) =  \tr V (C C^T) - t \sum_{k \in K} \log (x_k y_k).
$$
Here  $C$ is the $m \times m$ circulant version of the matrix $B(x,y)$.  Then, with  $\mathcal{H}_{*}$ the Hessian of $H_{K}$
at its minimizer $x_k = y_k = \phi(t)$ for $k \in K$, the needed fact is:
\begin{equation}
\label{eigcalc}
  v^T (\mathcal{H}_{*})^{-1} v = |K| \phi(t)^2 \frac{\theta'(t)}{\theta(t)},
\end{equation}
for  the vector $v$ with entries $v_k = (-1)^{k}, k=1,\dots 2|K|$. In particular, this is approximation of the 
$w^T \mathcal{H}_q^{-1} w^{}$ appearing in \eqref{VarG}. Comparing this to the statement  
 \eqref{VarCalc1}, note that there we sample $\phi$ and $\theta$ at the initial point $i_0/n$ rather than the midpoint $(i_0+i_1)/n$  $-$ the error in going back in forth between this two is easily seen to be of sufficiently lower order.

To see \eqref{eigcalc},  note that  $\mathcal{H}_{K}$ is circulant Toeplitz, and hence has $v$ as an eigenvector. With $\lambda$ the corresponding eigenvalue: with $(z_1, z_2, z_3, z_4, \dots) =
 (x_{k_0}, y_{k_0}, x_{k_0+1}, y_{k_0+1}, \dots)$, 
$$
   \lambda = \left( \sum_{ 1 \le k, \ell  \le 2|K|} (-1)^{k + \ell} \,  \frac{\partial^2 H_{K}}{\partial z_k \partial z_{\ell}} \right)  \Big|_{z_i  = \phi(t), i = 1, \dots, 2|K|} 
   = \frac{1}{2} \left( \frac{\partial}{\partial x }-  \frac{\partial}{\partial y} \right)^{2} H_t(x,y)   \Big|_{x = y =\phi(t)},
 $$ 
 where we bring back our initial definition of the coarse Hamiltonian $H_t$  from \eqref{coarseH}. Using that formula we find that
 $$
    \lambda = \frac{ t}{\phi(t)^2} + \sum_{m=1}^d g_m \frac{m}{2m-1} {2m \choose m} \phi(t)^{2m-2} = \frac{1}{\phi(t)} \int_0^t \frac{ds}{\phi(s)} = \frac{2 \theta(t)}{\phi^2(t) \theta'(t)}.
 $$
 The middle equality follows from the definition  \eqref{phi_def} for $\phi$ upon multiplying both sides of that identity  by $\phi'(t) \phi^{-2}(t)$ and integrating by parts. 
 
 It remains to show that
 \begin{equation}
      \left|  w^T \mathcal{H}_q^{-1} w -  v^T (\mathcal{H}_{*})^{-1} v \right| = O  \Bigl( (\log n)^2 \Bigr).
 \end{equation}
 Consider  the matrix $\mathcal{H}$ formed by setting all entries in the $O(1)$ blocks in 
 the upper right and lower left corners of $\mathcal{H}_{*}$ to zero. Clearly we have that
 $|   v^T \mathcal{H}^{-1} v -  v^T (\mathcal{H}_{*})^{-1} v | = O(1) $. On the other hand $ \mathcal{H}_q$ and $\mathcal{H}$ are also nearly the same. They are $2d$-banded with corresponding entries built from the same functions $-$ except along the diagonal $-$  evaluated  at either $(x^q, y^q)$ or $(\phi(t), \phi(t))$. Along the diagonal the functional entries differ only in 
the coefficients of the terms  corresponding to the second derivatives of the logarithm in $H_q$ or $H_{K,t}$:
there one must compare $t  \phi^{-2}(t)$ to $( k + a - 1/\beta) n^{-1} (x_k^q)^{-2}$ or $   ( k - 1/\beta) n^{-1} (y_k^q)^{-2}$
for any $k \in K$. But these coefficients ($ t$ and $( k + a - 1/\beta)/n$  or $(k-1/\beta)/n$) are no more than $O(|K| n^{-1}) = O(n^{-3/4})$ apart. Further, 
restricted to $Q_K$ the values of   $(x^q, y^q)$ and  $(\phi(t), \phi(t))$ 
are no more than $O(\sqrt{\frac{\log n}{n}})$ apart (and are uniformly bounded below).  It follows that
 \begin{equation}
 \label{GG}
   \| \mathcal{H}_q^{-1} - \mathcal{H}^{-1} \| = \| \mathcal{H}_q^{-1}\|  \|  \mathcal{H}^{-1} \|  \,  \| \mathcal{H}_q - \mathcal{H} \| \le c''  \sqrt{\frac{\log n}{n}}.
 \end{equation}
Convexity of $H_q$ provides a constant upper bound on $\|  \mathcal{H}_q^{-1}\|$ and  $\|  \mathcal{H}^{-1} \|$.
 Then, the previous remarks along with the  Gershgorin circle theorem yield 
 $ \| \mathcal{H}_q - \mathcal{H} \| = O(\sqrt{\frac{\log n}{n}})$.
   It remains to show  that
  \begin{equation*}
  \label{offblock}
     | (w-v)^T \mathcal{H} v | =  O  \Bigl( (\log n)^2 \Bigr),
 \end{equation*}
 and likewise for  $(w-v)^T \mathcal{H}^{-1} w$. Noting that $(w-v)$ has only $O(\log n)$ non-zero entries, this is a consequence of the banded Toeplitz nature of $\mathcal{H}$ which produces exponential decay in the entries of $\mathcal{H}^{-1}$ away from the diagonal. 
 \end{proof}

\begin{proof}[Proof of Corollary \ref{prop:Shifts}]
The idea is similar to that behind Proposition \ref{prop:testfunctions}, though now for each $k \in [c \log n, n- c \log n]$ we let  $K$ be the interval of length $c \log k$ centered at $k$ for  the  constant $c$   to be chosen momentarily.  

 Let again $q$ denote the coordinates in $\partial K$, but  now let $Q_{\partial K}$ be the event that
$  ||q - (x^o, y^o)||_{\partial K, \infty}$ is less than  ${c'}{\sqrt{ \frac{\log k}{n}}} $. By choice of $c'$ and 
Proposition \ref{prop:Concentration}, 
$P(Q_{\partial K}) = 1-  O (k^{-4})$ since $k \gg 1$.  The same proposition gives that
$E [ (X_k-x_k^o)^{2p} + (Y_k - y_k^o)^{2p} ]= O(n^{-p})$. Both estimates are uniform in $k$.  And so, by the Cauchy-Schwartz (and Jensen's) inequality $E [ E_q |X_k - x_k^o|^p , Q_{\partial K}^c ]  = O (n^{-p} k^{-2})$, and likewise in the $y$-variable.
Next, for any given $q$ in a given $Q_{\partial K}$, we can select the $c = c(\beta, V, a)$ so that  Proposition \ref{thm:BCdecay} gives $| x_k^o - x_k^q | + | y_k^o - y_k^q| = O( e^{-c''(c) \log k}  (\log k / n)^{1/2})$
$= O( k^{-2} n^{-1/2})$. 
 The conclusion is that: for $p=1$ or $2$,
\begin{align}
\label{lastGaussmove}
E \Bigl[ (X_k - x_k^o)^p - (Y_k- y_k^o)^p] & = 
  E \left[  E_q \Bigl[  \triangle_p(X_k, Y_k)  \one_{ Q_K} \Bigr]  \one_{Q_{\partial K}} \right] + O( n^{-p/2} k^{-2} ),
\end{align}
 uniformly  in $k$. Here we have made the definition, 
 $$
 \triangle_p(x_k, y_k) = (x_k - x_k^q)^p - (y_k- y_k^q)^p,
 $$ 
 and $Q_K = Q_K(q)$
 is the event $\{ \| (x, y) - (x^q, y^q) \|_{K, \infty} \le c'  {\sqrt{ \frac{\log k}{n}}} \} $, for a possibly adjusted $c'$.
 That we can restrict the $E_q$ integral in \eqref{lastGaussmove} to $Q_K$ with the stated level of error, follows from the same argument used at the analogous step in the proof of  Proposition \ref{prop:testfunctions}.

 Turning to an estimate on  $E_q [  \triangle_p(X_k, Y_k)  \one_{ Q_K}]$, we start with the case $p=2$. Under the approximating Gaussian measure $\nu_q = \nu_{K, q}$ we have that 
 \begin{align}
 \label{firstorders}
  %\int \Delta_1(x_k, y_k) \nu_q(dx, dy) & = 0,  \\
       \int_{Q_K}  \Delta_2(x_k, y_k) \nu_q(dx, dy) & = \frac{1}{n \beta}  \Bigl(  (\mathcal{H}_q^{-1})_{2k-1,2k-1}  - 
     (\mathcal{H}_q^{-1})_{2k, 2k} \Bigr)  + O \left( \frac{1}{n k^2}  \right).
\end{align}
The first term is an exact Gaussian computation, after removing the restriction to $Q_k$. The error term uses that by choice of $c'$ it holds $\nu_q ( Q_K) = 1 - O (k^{-4})$.  With $K = [k_0, k_1]$ the indices of the $\nu_q$ integral and the matrix $\mathcal{H}_q$ are indexed  $x_{k_0}, y_{k_0}, x_{k_0+1}, y_{k_0+1},  ...$, as before.

Next recall that the  proof of Corollary \ref{lem:VarEst} introduces  a banded Toeplitz approximate $\mathcal{H}$ to $\mathcal{H}_q $ ($\mathcal{H}$ is the Hessian of the coarse Hamiltonian $H_K$ on $K$, with the corner entries which make the latter circulant removed), and would like to replace the appearances of $\mathcal{H}_q$ in 
\eqref{firstorders} with this approximate. Since 
$(\mathcal{H}^{-1})_{ii}  =       (\mathcal{H}^{-1})_{jj}$ for all $i, j \in K$ and, by convexity,  
$\| \mathcal{H}_q^{-1} - \mathcal{H}^{-1} \|$ is controlled by a constant multiple of $\|  \mathcal{H}_q - \mathcal{H} \|$,  this norm must  now be estimated for $K$ possibly within $\log n $ of the singularity. 

In the current setting we have that: with $t = k/n$ and so $\phi(t)$ the common variable where the entries of
$\mathcal{H}$ are evaluated,  $|x_i^q - \phi(t)| + |y_i^q - \phi(t)|$ $= O( \sqrt{\log k / n})$. Hence the difference between any off diagonal of $\mathcal{H}$ and $\mathcal{H}_q$ are also controlled by $O( \sqrt{\log k / n})$. The more delicate issue is now the diagonals where one has to consider the absolute differences 
$| \phi(t)^{-2} - (x_i^q)^{-2}|$ or $|\phi(t) -   (y_i^q)^{-2}|$ which are $O ( (n/k)^{3/2} \times \sqrt{\log k /n})$. Here we use that in general we have that $\phi(t) \ge \delta \sqrt{k/n}$, and so the given $x_i^q$ and $y_i^q$
for $i \in K$ satisfy the same lower bound. But since any of these diagonal components are multiplied by coefficients 
which are $O(k/n)$ throughout $K$, the corresponding entry differences are actually $O ( \sqrt{\log k / k})$ and  we have that 
\begin{equation}
\label{newNormDif} 
 \| \mathcal{H}_q - \mathcal{H} \| = O \left( \sqrt{ \frac{\log k}{k}}   \right),
\end{equation} 
compare \eqref{GG}.
  Therefore, \eqref{firstorders} can be continued as in
\begin{equation}
\label{D2}
      \int _{Q_K} \Delta_2(x_k, y_k) \nu_q(dx, dy)  =  O \left( \frac{\sqrt{\log k}}{n \sqrt{k}} \right).
\end{equation} 

To finish, write 
\begin{equation}
\label{renorm1}
    E_q \Bigl[  \triangle_2(X_k, Y_k)   , \, Q_K \Bigr] =  \frac{Z_q^{\prime}}{Z_q} \int_{Q_K}  \Delta_2  d \nu_q + 
             \frac{Z_q^{\prime}}{Z_q}  \int_{Q_K} \Delta_2 ( e^{n\beta H_q^{(3)}} - 1) d\nu_q,
\end{equation}
where once more $Z_q^{\prime}$ and $Z_q$ denote the normalizers for $\nu_q$ and $P_q$, respectively.
Now recalling \eqref{3rdOrderB} from the proof of Proposition \ref{prop:testfunctions}, the estimate \eqref{Cheap3rdOrder} can be replaced by
\begin{equation}
 \label{NewCheap3rd}
   n   \beta | H_q^{(3)} (x, y) | \one_{Q_K} \le \gamma    \left( \frac{ (\log k)^{5/2}}{k^{1/2}} \right),
 \end{equation} 
 for another constant $\gamma = \gamma(V, \beta, a , c, c')$.
Here we have used the current definition of $Q_K$ which restricts $\| (x^q, y^q) - (x,y) \|_{K,\infty}^3$ to
$O( (\log k / k)^{3/2})$, that now $|K| = c \log k$, and a worse case upper bound
of $O( (n/k)^{3/2})$ on any of the $x_i(t)^{-3}$ or $y_i(t)^{-3}$ for $i \in K$ (recall that these are the interpolants from $(x_i^q, y_i^q)$ to $(x_i, y_i) \in Q_K$). Next, since the right hand side of \eqref{NewCheap3rd} is $o(1)$ for $k \ge c \log n$, \eqref{Zsandwhich} shows the ratios $Z_q^{\prime}/Z_q$ are bounded above and below by constants only depending on $c, c'$ and $V, \beta, a$. Finally then, using that $| e^\zeta  -1| \le 2 |\zeta|$ for $|\zeta| \le 1$
(applied to $\zeta = n \beta H_q^{(3)}$ restricted to $Q_K$) and $\int |\triangle_2| d\nu_q = O(n^{-1})$, we find for the second term in \eqref{renorm1} that
\begin{equation}
\label{T2Error}
   \int_{Q_K}  \Delta_2 ( e^{n\beta H_q^{(3)}} - 1) d\nu_q   = O \left( \frac{ (\log k)^{5/2}}{nk^{1/2}} \right).
\end{equation}
This is the estimate reported in \eqref{eq:VarShift}. 
%One expects that by taking better advantage of cancelations in \eqref{T2Error}  the proper conclusion is %that \eqref{D2} gives the correct behavior of $E[ (X_k - x_k^o)^2 - (Y_k -y_k^o)^2]$, but what we have %here will suffice.

For the difference of the means, one has to consider an additional order. We now write, 
\begin{equation*}
 E_q [ \triangle_1(X_k, Y_k),  Q_k] =  \frac{Z_q^{\prime}}{Z_q} \int_{Q_K}  \Delta_1 (1+ n H_q^{(3)} ) d \nu_q + 
             \frac{Z_q^{\prime}}{Z_q}  \int_{Q_K} \Delta_1 ( e^{n\beta H_q^{(3)}} - 1 - n H_q^{(3)}) d\nu_q,
\end{equation*}
for which we readily have the following: 
\begin{align}
\label{T1_2}
   \int_{Q_K}  \Delta_1  d \nu_q  & = O( n^{-1/2} k^{-2}), \\
   \int_{Q_K}   \Delta_1 ( e^{n\beta H_q^{(3)}} - 1 - H_q^{(3)})  d\nu_q   & = O \left( \frac{ (\log k)^{5}}{n^{1/2} k}  \right). \nonumber
\end{align}
The first of these is due: $\int \triangle_1 d \nu_q = 0 $, $\int |\triangle_1|^2 d \nu_q  = O(n^{-1})$, while
$c'$ can be chosen so that $\nu_q(Q_K^c) = O(k^{-4})$. (The displayed estimate follows from applying Cauchy-Schwartz to integral over $Q_K^{c}$). The second is similar to \eqref{T2Error}: now $\int   |\triangle_1| d \nu_q = O(n^{-1/2})$ while $| e^{n\beta H_q^{(3)}} - 1 - n H_q^{(3)}|$ on $Q_K$ is controlled by the square of the right hand side of \eqref{NewCheap3rd}.

As  we have noted,  $Z_q^{\prime}/Z_q$ is of constant order (uniformly for all $K$ and choices of ``good" boundary conditions $q$), and so it remains to consider 
 $ n \int_{Q_K} \triangle_1 H_q^{(3)} d \nu_q$. For this we first schematize $H_q^{(3)}$ as in 
\begin{equation}
\label{H3redux}
   H_q^{(3)}(x,y) = P_{V}(x,y) + \sum_{i \in K}  \Bigl( a_i (x_i - x_i^q)^{3}  + b_i (y_i - y_i^q)^3 \Bigr). 
\end{equation}
Here  $P_{V,q}$ represents the appropriate sum of third derivatives of the potential term,  while with 
$$
a_i = 2 x_i(t)^{-3}(i/n + a/n - 1/n \beta), \quad b_i = 2 y_i(t)^{-3}(i/n - 1/n \beta),
$$
the sum over centered cubics corresponds to the  (third derivatives of the) logarithmic terms of $H_q$. Since $a_i$ and $b_i$ are complicated functions of $(x,y)$, to perform the desired integral we first note:
on $Q_K$,
$$
   |a_i - c_i | +  |b_i  - c_i|  = O \left(\frac{ \sqrt{n \log k}}{k}  \right), \quad \mbox{ with } c_i = \frac{2i}{n \phi^3(i/n)}. 
$$
Hence, if we replace all appearances of $a_i$ and $b_i$ in  \eqref{H3redux} by $c_i$, we make an $O ( \log k \times    \frac{ \sqrt{n \log k}}{k}  \times \frac{(\log k)^{3/2}}{n^{3/2}} ) $ $= O(\frac{(\log k)^2}{n k} )$ sup-norm error  (granted we working on $Q_K$) to $H_q^{(3)}$, and  so a $O(\frac{(\log k)^2}{\sqrt{n} k} )$ error in any estimate of
$ n \int_{Q_K} \triangle_1 H_q^{(3)} d \nu_q$.

Similar preprocessing is required for the integral involving $P_V$. However, that integral will clearly be subdominant compared  with that over the second term in \eqref{H3redux} since  both $a_i$  and $ b_i$ are as large as $O(\sqrt{n/i})$. We will therefore only detail how to deal with this term.  After making the substitution just described, we have the evaluation:
\begin{align}
\label{T1_1} 
 \int  \triangle_1 (x_k, y_k) & \left( n  \sum_{i \in K} c_i \left( (x_i - x_i^q)^3  + (y_i - y_i^q) \right) \right) d \nu_q  \\
   & =   \frac{1}{n} \sum_{i \in K} (3  c_i)  \left[
       (\mathcal{H}_q^{-1})_{2i-1,2i-1} \Bigl(  (\mathcal{H}_q^{-1})_{2k-1,2i-1}  -  
     (\mathcal{H}_q^{-1})_{2k, 2i-1} \Bigr)  \right. \nonumber   \\
&      \hspace{2.2cm}  +  \left.
       (\mathcal{H}_q^{-1})_{2i,2i} \Bigl(  (\mathcal{H}_q^{-1})_{2k-1,2i}  -  
     (\mathcal{H}_q^{-1})_{2k, 2i} \Bigr) \right]. \nonumber
\end{align}
It is by now understood that we can go from this full-space integral to that restricted to $Q_K$ making further subdominant errors. 
Things are at long last wrapped in the same way that \eqref{firstorders} was treated. First observe that, if we could replace  
$\mathcal{H}_q$ with its approximate $\mathcal{H}$ throughout \eqref{T1_1}, the quantity within the square brackets vanishes on account that
$\mathcal{H}$ is Toeplitz.  Since again all entries  of $\mathcal{H}_q^{-1}$ and $\mathcal{H}^{-1}$ are uniformly bounded, a computation using
 \eqref{newNormDif}  shows that the 
the error incurred in making that substitution in \eqref{T1_1} is $O ( \frac{(\log k)^{7/2}}{ \sqrt{n} k} )$. As this lies under the larger of the error estimates in \eqref{T1_2} $-$ which is what is reported in \eqref{eq:MeanShift} $-$ the proof is finished.
\end{proof}

\section{Central limit theorem}
\label{sec:CLT}

Here we complete the identification of the limit of the $K_n$ kernel by proving:

\begin{proposition}
\label{prop:CLT}
As $n \rightarrow \infty$, 
\begin{equation*}
%\label{path_convergence}
   \sum_{k= \lfloor n t \rfloor}^n \log \frac{X_k/x_k^o}{Y_k/y_k^o} \Rightarrow \frac{1}{\sqrt{\beta}} 
   \int_{\theta(t)}^1 \frac{d b_u}{\sqrt{u}},
\end{equation*}
in the Skorohod topology on $(0,1]$.
\end{proposition}

Recall \eqref{path_convergence}.  Note that Gaussian concentration plus the formulas for the minimizer developed in the last two sections already give that $ X_{\lfloor nt \rfloor} \Rightarrow \phi(t)$ as processes on $[\delta, 1]$ for any $\delta >0$.

\subsection{Linearizing}

As a first step we have the following.

\begin{lemma} 
For Proposition \ref{prop:CLT} it is sufficient to show that
\begin{equation*}
%\label{path_convergence}
   \sum_{k= \lfloor n t \rfloor}^n  \frac{(X_k-x_k^o)  - (Y_k - y_k^o) }{\phi(k/n)} \Rightarrow \frac{1}{\sqrt{\beta}} \int_{\theta(t)}^1 \frac{d b_u}{\sqrt{u}},
\end{equation*}
in the Skorohod topology on $(0,1]$.
\end{lemma}

\begin{proof} We fix a (small) $\delta >0$, and show the claim for all processes restricted to $t \in [\delta, 1]$. Afterwards it will be clear that the choice of $\delta$ is arbitrary.

 Again denote
$$
  Q = \left\{  | X_k - x_k^o |, | Y_k - y_k^o |  \le c  \sqrt{\frac{\log n}{n}}  \mbox{ for all  }k  \in [1,n] \right\},
$$
with $c$ chosen so that $P(Q^{c}) \le n^{-4}$ for all $n$ large enough (Proposition \ref{prop:Concentration}). 
Certainly it is enough to work with the process $ \one_Q  \sum_{k= \lfloor n t \rfloor}^n \log \frac{X_k/x_k^o}{Y_k/y_k^o}$. At the same time,
$$
\one_Q  \left[  \log \frac{X_k/x_k^o}{Y_k/y_k^o} - 
   \left( \frac{(X_k - x_k^o)  - (Y_k - y_k^o) }{\phi(k/n)}  \right)  +    \left( \frac{(X_k - x_k^o)^2   - (Y_k - y_k^o)^2}{2 \phi(k/n)^2} 
    \right)  \right] = O (n^{-(3/2-\epsilon)}),
$$
uniformly for $k \in [n \delta, n]$ with probability one. Here $\epsilon$ can be chosen as small  as one likes subject to the implied constant on the right hand side depending on $\epsilon$. This follows as $| \log (1+t) - t + t^2/2 | \le |t|^3$ for all $|t| \le 1/2$ while, with $(Z,z) = (X,x)$ or $(Y,y)$: on $Q$ all  $| Z_k - z_k^o|^p$ are $O(n^{-(p/2-\epsilon)})$, $  z_k^o$ is uniformly bounded below for  $k > \delta n$, and $| \frac{1}{z_{k}} - \frac{1}{\phi(k/n)} | = O(n^{-1})$ throughout the same range of indices. For the last two facts see Proposition \ref{prop:FineMin}. So, with the
left hand side of the above display denoted by $\eta_{k, n}$ we have that $ t \mapsto   \sum_{\lfloor nt \rfloor}^n \eta_{k,n}$ converges to the zero process.

Consider next
$$
   \zeta_m^n(X,Y) = \one_Q \sum_{k= m}^n   \frac{(X_k - x_k^o)^2   - (Y_k - y_k^o)^2}{ \phi(k/n)^2}.
$$
The proof will be finished by showing that $\max_{m \in [\delta n, n]} | \zeta_m^n |$ goes to zero with probability one.
We actually take the approximation
$$
   \hat{\zeta}_m^n(X,Y) = \sum_{k= m}^n \frac{ \psi( X_k - x_k^o)   - \psi(Y_k - y_k^o) }{\phi(k/n)^2},
$$
for $\psi(z) = z^2$ for $|z| \le c \sqrt{\frac{\log n}{n}}$ outside of which $\psi$ is taken to be constant. Obviously, $\hat{\zeta}_m^n$ and $\zeta_m^n$ agree on $Q$, while as a map from $x_m, y_m, \dots x_n, y_n-1$ to $\RR$, we have that
$ | \nabla \hat{\zeta}_m^n (x, y) |^2 \le c' \log n$ with $c'$ depending on $\delta$.  Lemma \ref{lem:LSI} then implies that
\begin{equation}
\label{eq:QuadKill}
  P \left( | \hat{\zeta}_m^n - E \hat{\zeta}_m^n | > n^{-1/4} \right) \le  2 e^{ - c'' ( n^{1/2}/ \log n )}, 
\end{equation}
for all $m > \delta n$ with a $c'' = c''(V,\beta,a, \delta)$.
But then by \eqref{eq:VarShift} of Corollary \ref{prop:Shifts}   have that $E \zeta_m^n = O ( (\log n /n)^{1/2})$ uniformly in $m > \delta n$, and the same estimate will hold for $E \hat{\zeta}_m^n$. Now the result follows from  \eqref{eq:QuadKill} and a union bound.  
\end{proof}

\subsection{Finite dimensional convergence}

We employ a classical blocking argument, with the limit being understood through the sum over ``good blocks" (of length $O(n^{1/6}))$ of the variables, each  such block
separated by $O(\log n)$ ``buffers". That the minimizer of any conditional Hamiltonian becomes independent of the boundary in $O(\log n)$ steps will produce the required decorrelation between adjacent blocks.

Define recursively the times,
\begin{equation*}
   m_1 =1, \quad  m_{k+1} =  \left\{ \begin{array}{ll} m_k + \lfloor c \log n \rfloor &  \mbox{ if } k \mbox{ is odd} \\
                                        m_k + \lfloor n^{1/6} \rfloor &  \mbox{ if } k \mbox{ is even},  \end{array} \right.
\end{equation*}
and corresponding good blocks and buffers: for $i = 1, 2, \dots$,
\begin{align}
\label{eq:Blocks}
   \mathcal{G}_i &  = \frac{1}{\phi(m_{2i}/n)} \sum_{k = m_{2i}}^{m_{2i+1}} [(X_k - x_k^o)   - (Y_k - y_k^o)], \\
   \mathcal{B}_i  & = \frac{1}{\phi(m_{2i-1}/n)} \sum_{k = m_{2i-1}}^{m_{2i}} [(X_k - x_k^o)   - (Y_k - y_k^o)]. \nonumber
\end{align}
Here we made one more approximation in pulling the $\phi^{-1}$ of smallest index out of each block sum. By the continuity of $\phi$ it will be clear that this will make no difference in what follows. Also, truncating the final $\mathcal{G}_i$ sum if necessary we can always assume that the stretch $[n-c \log n, n]$ is buffer. With this setup the result is:

\begin{lemma}
\label{lem:FiniteDim}
Set $\mathcal{G}_t^n = \sum_{i: m_i \in [ nt, n]} \mathcal{G}_i$ and 
$\mathcal{B}_t^n = \sum_{i: m_i \in [ nt, n]} \mathcal{B}_i$.  Then as $n \rightarrow \infty$, there is a suitable choice of
$c = c (V, \beta, a)$ in \eqref{eq:Blocks} so that for any $k$ and 
$ 0 < t_1 < t_2 < \cdots < t_k \le 1$, 
$ (\mathcal{G}_{t_1}^n, \mathcal{G}_{t_2}^n, \dots, \mathcal{G}_{t_k}^n) $ and 
   $  (\mathcal{B}_{t_1}^n, \mathcal{B}_{t_2}^n, \dots, \mathcal{B}_{t_k}^n),  $
 converge in law to a centered Gaussian vector with covariance $\frac{1}{\beta} \log \frac{1}{\theta(t_i)} \wedge
\frac{1}{\beta} \log \frac{1}{\theta(t_j)} $ and the zero vector, respectively.
\end{lemma}

\begin{proof}
%[Proof of Lemma \ref{lem:FiniteDim}]
We start by estimating $E e^{i \tau \mathcal{G}_t^n}$ for $t$ fixed. With $I_i$ the support of any corresponding  $\mathcal{G}_i$ figuring into $\mathcal{G}_t^n$, denote by $K_i$ the interval formed by adjoining the $(c/3) \log n$ length stretches of indices to the left/right of $I_i$. The parameter $c$ is chosen large enough so that the strategy of 
 Proposition \ref{prop:testfunctions}  can be followed (the length $c \log n $ buffer about $I_i$ is now length 
 $(c/3) \log n$, but $c$ is chosen as needed in both cases). In particular, boundary values set at $\partial K_i$ will have weak influence on the statistics of $\mathcal{G}_i$.
 
 With $q_i$ the variables in $\partial K_i$ we by now understand that
 \begin{equation}
 \label{charfunct1}
   E e^{i \tau \mathcal{G}_t^n} = E \Bigl[ \one_Q  \prod_{i : m_i \in [ nt, n]} E_{q_i}  \left[ e^{i \tau \mathcal{G}_i}  \right]\Bigr] + o(1)
 \end{equation}
 where $Q$ is the event that $\max_{i} \| q - (x^o, y^o) \|_{\partial K, \infty} \le c' 
 \sqrt{ \frac{\log n}{n}}$ for suitably large $c'$. Further,
 $$
  \Bigl| E_{q_i} [ e^{i \tau \mathcal{G}_i} - (1 + i \tau    \mathcal{G}_i - \frac{1}{2}  \tau  \mathcal{G}_i^2 ) ] \Bigr| \le 
  E_{q_i} | \tau \mathcal{G}_i|^3 = O(|I_i|^3 n^{-3/2}),
 $$
by Proposition \ref{prop:testfunctions} (or rather its proof). Similarly, by Corollary \ref{prop:Shifts}
we have that
$E_{q_i}   \mathcal{G}_i  = O ( |I_i| (\log n)^5 n^{-3/2} ).$ Combining these facts with Corollary \ref{lem:VarEst} we find that
\begin{equation}
\label{charfunct2}
    E_{q_i} e^{i \tau \mathcal{G}_i} = 1 -  \frac{\tau^2}{2} \frac{\theta'(m_{2i}/n)}{ \beta \theta(m_{2i}/n)}   
      \frac{(m_{2i+1}- m_{2i})}{ n} + \kappa_n
\end{equation}
with $|\kappa_n| = o(   (m_{2i+1}- m_{2i}){ n}^{-1} )$ uniformly in $q_i \in Q$. Substituting back into \eqref{charfunct1} we recognize the Riemann sum for $\int_t^1 \theta^{\prime}(s) \theta^{-1}(s) ds$  on scale $\triangle =  (m_{2i+1}- m_{2i}){ n}^{-1} $. 

The same considerations apply to $\mathcal{B}_t^n$, with the right hand side of \eqref{charfunct2} modified by shifting $2 i$ to $2i-1$.
But  $ (m_{2i}- m_{2i-1}){ n}^{-1} = o(\triangle)$ while there are still $\triangle^{-1}$ factors in the analog of \eqref{charfunct1}
The outcome is that $E  e^{i \tau \mathcal{B}_t^n} \rightarrow 1$ as $n \rightarrow \infty$.   To be precise we note that while
Corollaries  \ref{lem:VarEst}  and \ref{prop:Shifts} do not apply to the final block $\mathcal{B}_{i(n)}$ in $\mathcal{B}_t^n$ (as it is constructed to be supported on $[n - c \log n , n]$, a more crude estimate by Proposition \ref{prop:testfunctions} gives that
$E (\mathcal{B}_{i(n)})^2 = o(1)$. 

%Note that while Lemma \lem{lem:VarEst} and Corollary \ref{prop:shifts} do not apply to the last stretch of.... 

The convergence of $k$-point marginals follows from the asymptotic independence of increments for $t \mapsto \mathcal{G}_t^n$ which is immediate from  the necessary version of  \eqref{charfunct1}.  Taking $k=2$ gets the point across.  With $s < t$ and any $\tau $ and $\nu$,
\begin{align*}
E e^{i \tau \mathcal{G}_s^n + i \nu \mathcal{G}_t^n}  & =  E \left[  E_{q} e^{ i \tau (\mathcal{G}_s^n - \mathcal{G}_t^n )} 
  E_{q'} e^{i (\tau + \nu) \mathcal{G}_t^n } \right] \\
    & = e^{ - \frac{\tau^2}{2\beta} \log \frac{\theta(t)}{\theta(s)} } e^{ - \frac{(\tau + \nu)^2}{2 \beta} \log \frac{1}{ \theta(t)}} 
       +o(1), 
\end{align*}
and the exponent reads $ (- \frac{1}{2 \beta}) \times $ $(\tau^2 \log \frac{1}{\theta(s)} +   \nu^2 \log \frac{1}{\theta(t)}  + 2  \tau \nu \log \frac{1}{\theta(t)} )$ as desired. In line one above we simply use  that no $\mathcal{G}_i$ is included in both 
$(\mathcal{G}_s^n - \mathcal{G}_t^n)$ and $  \mathcal{G}_t^n$ $-$ the convention being it belongs to the sum in which its left-most point of support lies. Thus the conditionings variables  $ q$ and $q'$ can be chosen not to overlap, and to be a distance $O(\log n)$ from any of the corresponding  $\mathcal{G}_i$'s within. Now we simply apply the strategy inherent in
\eqref{charfunct2} to the $E_q$ and $E_{q'}$ expectations separately.
\end{proof}

\subsection{Pathwise convergence}

To lift the convergence from marginal distributions to convergence in the space of continuous paths we show the following.

\begin{lemma} 
\label{lem:Tightness}
With $ \zeta_t^n$ either equal to $\mathcal{G}_t^n$  or $\mathcal{B}_t^n$ defined in the statement of Lemma \ref{lem:FiniteDim}
and all $0 \le r \le s \le t \le 1$:
\begin{equation}
\label{tightmoments}
   E \left[ ( \zeta_r^n - \zeta_s^n )^2 ( \zeta_s^n - \zeta_t^n)^2 \right]  \le c \left( \log \frac{\theta(t)}{\theta(r)} \right)^2
\end{equation}
for a constant $c$ and all large enough $n$.
\end{lemma} 

This suffices for the tightness of $t \mapsto \mathcal{G}_t^n$ and $t \mapsto \mathcal{B}_t^n$  due to Theorem 13.5 of \cite{Bill}. To compare the above with that statement, note that we are using the latter in the case that the limit process is continuous, the $\alpha$ and $\beta$ parameters defined there equal to one, and choice of $F(t) =  - \sqrt{c}\log \theta(t)$ (our time-like parameter naturally runs ``in reverse").  We can then conclude the 
full convergence of the desired process $t \mapsto \mathcal{G}_t^n +  \mathcal{B}_t^n$ to $ t \mapsto  \int_{\theta(t)}^1 \frac{d b_u}{\sqrt{\beta u}}$  
by Slutsky's Lemma.

\begin{proof} We show the inequality \eqref{tightmoments} holds for the process of good blocks $\mathcal{G}_t^n$. A similar calculation will apply to  $\mathcal{B}_t^n$. To start write 
\begin{equation}
\label{TBound1}
   E \left[ (  \mathcal{G}_r^n -  \mathcal{G}_s^n )^2 (  \mathcal{G}_s^n -  \mathcal{G}_t^n)^2 \right]  
   = \sum_{\substack{i_1, i_2: m_{2i} \in [nr, ns)  \\ j_i, j_2: m_{2j} \in [ns, nt)}} E [ \mathcal{G}_{i_1} \mathcal{G}_{i_2} \mathcal{G}_{j_1} \mathcal{G}_{j_2} ],
\end{equation}
recalling \eqref{eq:Blocks}.    The main contribution stems from terms on the right hand side of \eqref{TBound1} in which 
$i_1 = i_2$ and $j_1 = j_2$. For any such term we have that: with $\delta_n = c' \sqrt{\frac{\log n}{n}}$ and large enough $c'$,
\begin{align}
\label{TBound2}
   E [  \mathcal{G}_{i}^2 \mathcal{G}_{j}^2 ]  & = E \Bigl[   E_{q_i} [ \mathcal{G}_{i}^2 ]  \,  E_{q_j} [ \mathcal{G}_{j}^2 ] , \| q_i, q_j  \|_{\infty} \le \delta_n  \Bigr] + o(n^{-2}) \\
   & \le c''  \frac{(m_{2i+1} - m_{2i}) (m_{2j+1} - m_{2j})}{ n^2 \phi(m_{2i}/n) \phi(m_{2j}/n)} + o(n^{-2}), \nonumber
 \end{align}
  by reasoning used several times before. And as  in all such cases, we can choose the supports of the  disjoint boundary $q_i$ and $q_j$ a large multiple of $\log n$ away from the respective supports of $ \mathcal{G}_{i}$ and $ \mathcal{G}_{j}$. 
The inequality in \eqref{TBound2}  is then a direct consequence of Corollary \ref{lem:VarEst}. Summed over all 
$O(n^2  | m_{2k+1} - m_{2k}|^{-2}) = O( n^{5/3} )$ possible $i$ and $ j$ we get a constant multiple of 
$\int_r^s \frac{du}{\phi(u)} \int_s^t \frac{dv}{\phi(v)}$
upper bound for the corresponding subsum of \eqref{TBound1}.

Terms of type $i_1 \neq i_2$ or $j_1 \neq j_2$ in \eqref{TBound1} are easily seen to be subdominant given, in this regime of indices,   $ |  E_q [ \mathcal{G}_i ] | \one_{|| q||_\infty \le \delta_n} = O(n^{-(3/2-)})$.  This is a byproduct of the proof of 
Corollary \ref{prop:Shifts}.
\end{proof}

\section{Convergence in norm}
\label{sec:Norm}

The results of the previous section imply the pointwise convergence of $K_n(s,t)$ to $K(s,t)$, at least over subsequences on a suitable probability space.  The proof of the convergence of the corresponding operators in Hilbert-Schmidt norm (and in the same subsequential coupling) would follow if we could build a dominating kernel $\widehat{K}$ (that is, $K_n(s,t) \le \widehat{K}(s,t)$) which lies almost surely in $L^2 ([0,1]^2)$. Note that one readily checks that 
$\int_0^1 \int_0^t | K (s,t)|^2 \, ds dt < \infty$ with probability one.

The next proposition provides such an estimate, but only away from the singularity at the origin. 

\begin{proposition}
\label{prop:kernelbound}
For sufficiently large $c = c(V, \beta, a)$ and  any $\epsilon > 0$,
\begin{equation}
\label{ineq:kernelbound} 
  K_n(s,t) \le C_n \frac{  ( \theta(s) \theta(t) )^{-\epsilon} }{ (\phi(s) \phi(t) )^{1/2}} \left( \frac{\theta(s) }{\theta(t)} \right)^{\frac{a}{2} + \frac{1}{4}} 
  \quad \mbox{ for }  \  c \frac{\log n}{n} \le s \le t \le 1,  
\end{equation}
in which $C_n = C_n(c, \epsilon)$ is a tight random sequence.
\end{proposition}

The point is that, denoting the deterministic part of the right hand side of \eqref{ineq:kernelbound} by
\begin{equation}
\label{boundingkernel}
K_{\epsilon}(s,t) =   \frac{  ( \theta(s) \theta(t) )^{-\epsilon} }{ (\phi(s) \phi(t) )^{1/2}} \left( \frac{\theta(s) }{\theta(t)} \right)^{\frac{a}{2} + \frac{1}{4}}, 
\end{equation}
a calculation shows that  $\int_0^1 \int_0^t | K_{\epsilon} (s,t)|^2 ds dt < \infty$ as long as $\epsilon < \frac{1}{4} \wedge  \frac{(a+1)}{2}$.  To bridge the gap for small values of $s$ and $t$, we will  show the following.

\begin{proposition} 
\label{cheapprop}
For any $c>0$,
\begin{equation}
\label{vanishonstrip}
   \iint\limits_{ 0 \le s \le t \le 1, \, s \le c \frac{\log n}{n} } | K_n(s,t)|^2 \, ds dt \rightarrow 0
\end{equation}
 in probability.   
\end{proposition}

We mention that  in establishing the convergence of the classical $\beta$-Laguerre matrix model to 
$\mathrm{SBO}_{\beta,a}$ in \cite{RR} a single dominating kernel was relatively easy to come by. On the other hand, for the ``spiked" hard edge considered in \cite{RR1} in which case one deals with matrix kernel operators a similar cutoff procedure was required $O(\log n)$ steps away from the singularity.

In any case, one may now argue as follows.  Given any subsequence of operators $K_n$, choose a further subsequence $K_{n'}$ and a probability space on which \eqref{vanishonstrip} takes place almost surely and the bound \eqref{ineq:kernelbound} holds almost surely with the tight sequence $C_{n'}$ replaced by some deterministic constant (bounding the chosen subsequential limit of tight the prefactors).  Presuming the pointwise convergence $K_{n'} \rightarrow K$ of kernels 
also takes place almost surely on the same space (which may be achieved by taking yet a further subequence), it follows that
$ \int_0^1 \int_0^t | K_{n'} (s,t) - K(s,t) |^2 ds dt \rightarrow 0$ with probability one.  This completes the proof of Theorem \ref{thm:main} and hence the main result.

The proofs of Propositions \ref{prop:kernelbound} and \ref{cheapprop} occupy the next two subsections. 

\subsection{Tight kernel bound away from the singularity}

Proposition \ref{prop:kernelbound} is a consequence of the following.

\begin{lemma}
\label{LILtype}
Define $h(t) =  (1 + \log  \frac{1}{\theta(t)} )^p$. Then for any $p \in (1/2, 1)$ and $c = c(V, \beta, a)$ large enough, the sequence 
\begin{equation}
\label{eq:LIL}
       \max_{k \ge c \log n }  \frac{1}{h(k/n)} \sum_{j = k}^n \log  \left( \frac{Y_j/y_j^o}{X_j/ x_j^o } \right)   
\end{equation}
is tight.
\end{lemma}

The results of Sections \ref{sec:Min} and \ref{sec:Gauss} (concentration about the minimizers and proximity of the coarse and true minimizers) show that
$ \max_{k > c \log n} (\phi(k/n)/X_k)$ is also tight, controlling the prefactor to the random product 
appearing in the definition of $K_n(s,t)$. In particular then there is  a tight random sequence $C_n$ such that
$  1/X_k \le C_n /\phi(k/n) $ and $ \sum_{j=k}^n   \log  ( \frac{Y_j/y_j^0}{X_j/ x_j^0 } )  \le C_n h(k/n)$, at least for $k$ in the prescribed range. But that means that
\begin{align}
\label{Kagain}
 K_n(s,t) &  = X_{\lfloor nt \rfloor}^{-1}  \, 
 e^{ \sum_{k  =  \lfloor ns \rfloor}^{\lfloor nt \rfloor - 1 }   \log  \Bigl( (Y_k/y_k^o) /  (X_k/ x_k^o ) \Bigr) }    \, 
 e^{-  \sum_{k =   \lfloor ns \rfloor}^{\lfloor nt \rfloor - 1 }  \log \Bigl( x_k^o/y_k^o \Bigr) } \nonumber \\  
 & \le C_n  \frac{  e^{C_n (h(s)+ h(t) )} }{ (\phi(s) \phi(t) )^{1/2 + \epsilon}} \left( \frac{\theta(s) }{\theta(t)} \right)^{\frac{a}{2} + \frac{1}{4}} , \quad \mbox{ for } s,t > c \frac{\log n}{n}.
\end{align}
Here we have used that: 
\begin{equation*}
 \sum_{j= \lfloor nt \rfloor}^n \log  ( x_j^o/y_j^o ) \le (a/2+ 1/4) \log \theta(t) - 1/2\log \phi(t) +  c' (1 + \sqrt{\log \phi(t)} )
\end{equation*}
for $nt > c \log n$ and some constant $c'$, recall Proposition \ref{prop:FineMin}. This last (error) term is  then absorbed  into the $C_n \phi(t)^{\epsilon} $ in \eqref{Kagain}. That estimate in turn implies the claimed inequality \eqref{ineq:kernelbound} of  Proposition \ref{prop:kernelbound} as for any positive $c$ and $\epsilon$ and $p \in [0,1)$ there is a $c'' = c(\epsilon, p)$ so that $c (1+a)^p \le c' + \epsilon a$ for all $a \ge 0$. Afterwards $\epsilon$ is adjusted. Note here and below we write $K_n(s,t) =
 X_{\lfloor nt \rfloor}^{-1} 
\prod_{k= \lfloor ns \rfloor}^{ \lfloor nt \rfloor -1}
(Y_k /X_k)$. This is not quite accurate, compare the definition \eqref{discreteK}, but can be considered a convenient shorthand and will not make any difference for the level of estimate required here.

As for Lemma  \ref{LILtype} we need one last ingredient. This is due to Dudley  \cite{Dud} (though see Proposition 2.2.10 of \cite{Tal} for a succinct proof).

\begin{proposition}
\label{prop:Dudley}
Consider a metric space $(T, d)$ and  a centered process $(Z_t)_{t\in T}$ with law $\mathbf{P}$ satisfying
\begin{equation}
\label{Dud1}
  \mathbf{P} ( |Z_s - Z_t | \ge \lambda ) \le 2 e^{-\frac{ \lambda^2}{2 d(s,t)^2}}
\end{equation}
for all $\lambda > 0$. Then there is a universal constant $c$ such that
\begin{equation}
\label{Dud2}
  \mathbf{E} \sup_{t \in T} Z_t \le c \sum_{q \ge 0} 2^{q/2} e_q(T),
\end{equation}
in which $e_q(T) = \inf \sup_{t \in T} d(t, T_q)$ and the infimum is over all $T_q \subset T$ of cardinality $\le 2^{2^q}$.
\end{proposition}

\begin{proof}[Proof of Lemma \ref{LILtype}] The first step is to truncate the logarithm.
With $c' > 1$  to be chosen momentarily let
$\delta_n = c' \sqrt{\frac{\log n}{n}}$, and for each  $j \in [ c \log n , n ]$, where $c \ge c'$ will also be chosen along the way, define:
\begin{equation}
 \label{eq:cutoff}
   G_j(z)  =  \left\{  \begin{array}{ll}  \log(z/z_j^o),  & \mbox{ for } |z - z_j^o| \le \delta_n,  \ \\ 
                                                         \log(1- \delta_n/z_j^o) \mbox{ or } \log(1 +    \delta_n/z_j^o ),  & \mbox { for }
                                                         z \le z_j^o - \delta_n \mbox{ or } z \ge  z_j^o + \delta_n. \end{array}   \right.
\end{equation}
Here $z_j^o$ denotes the coordinate of the minimizer $x_j^o$ or $y_j^o$ according to whether $G_j$ is to be evaluated at $x_j$ or $y_j$.  Since both $x_j^o$ and $y_j^o$ can be bounded below by a small constant multiple of $\phi(j/n)$ which in turn is $O(\sqrt{j/n})$ for small $j$, by choice of $c = c(c')$ we have that: 
\begin{equation}
\label{eq:minlimbds}
\delta_n/x_j^o \vee \delta_n /y_j^o \le \frac{1}{2},  \quad (x_j^o - \delta_n) \wedge (y_j^o - \delta_n) \ge  \frac{1}{2} \sqrt{ \frac{j}{n}},  \mbox{ for all } j \ge c \log n.
\end{equation}
These bounds at least guarantee that \eqref{eq:cutoff} is sensible. Further, for all $k$ in the range of interest the sum
\begin{equation}
\label{eq:newsum}
  S_k = S_k(X,Y) = \sum_{j=k}^n G_j(Y_j) - G_j(X_j)
\end{equation}
agrees with $\sum_{j=k}^n \log  \left( \frac{Y_j/y_j^o}{X_j/ x_j^o } \right)  $ on the event $Q = \{ |X_j - x_j^o |, |Y_j - y_j^o| \le \delta_n \mbox{ for all } j \ge  \log n \}$. And then once again Proposition \ref{prop:Concentration} along with a union bound  implies $P(Q) = 1 - o(1)$ granted  
$c' = c'(V,\beta, a)$  is chosen large enough. Hence, to prove the claim it suffices to show that
\begin{equation}
\label{eq:newclaim}
    \max_{k \ge c \log n} \frac{S_k}{h(k/n)}  \quad \mbox{ is tight}.
\end{equation}
Where now $c = c(V, \beta, a)$ is fixed after an appropriate choice of $c'$.

Next we note that, as a map taking $(x_k,\dots x_n; y_k, \dots y_n) \mapsto \RR$, $S_k$ has square Lipschitz norm bounded as in
\begin{equation}
\label{SLip}
   \| \nabla S_k(x,y) \|_{2}^2 \le \sum_{j=k}^n \frac{1}{(x_j^o - \delta_n)^2} + \frac{1}{(y_j^o - \delta_n)^2} \le 
   8 n \sum_{j=k}^n  \frac{1}{j} \le 16 n \log (n/k),
\end{equation}
where the second inequality in \eqref{eq:minlimbds} is used. Therefore, by Lemma \ref{lem:LSI}
\begin{equation}
\label{SLSI}
   P \Bigl(    S_k(X,Y)  \ge E S_k(X,Y)  +  \lambda \Bigr)   \le  \exp { \left( -  \frac{\lambda^2}{ c'' \log \frac{n}{k} }   \right)},
\end{equation}
for all $\lambda > 0$ and  all $k \ge c \log n$ with yet another constant $c'' = c''(V, \beta,a)$

Turning to \eqref{eq:newclaim}, we introduce 
\begin{equation}
   F_m = \max_{e^{-m-1} <  \frac{k}{n} \le e^{-m}}  S_k, \quad  \mbox{  for } m = 1, 2, \dots, 
\end{equation}
and, noting that  it is only the small (as in $o(n)$) values of $k$ which really require attention, 
estimate as follows:
\begin{align}
\label{breakupbound}
 P \left( \max_{ c \log n \le k \le n/4}  \frac{S_k}{h(k/n)}  > \lambda \right) 
 & \le  \sum_{m=1}^{\infty} P \left(  F_m >  \lambda h( e^{-m}) \right) \\
 & \le \sum_{m=1}^{\infty} \exp \left( - (\lambda m^p  - E F_m )^2/ 4 c'' m \right). \nonumber
\end{align} 
The second inequality is due to \eqref{SLSI} along with the fact that the maximum function has Lipschitz norm one. 
Recalling that $ p> 1/2$, the proof will be finished by showing that $E F_m$ is uniformly bounded in $m$ and $n$  (in which case the right hand side above can be made as small as one likes by taking $\lambda  \uparrow \infty$). This is where Proposition
\ref{prop:Dudley} comes in.

Another application for \eqref{SLSI} shows that the condition \eqref{Dud1} of that Proposition is satisfied with the discrete process
$k \mapsto S_k - ES_k$ in the role of $t \mapsto Z_t$ with $T =[n e^{-m-1}, n e^{-m}]$ and metric $d(k,\ell)$ equal to a constant  ({\em{i.e.}}, independent of both $n$ and $m$) multiple of
$
  \sqrt{  | \log k/\ell | }. 
$
 In particular, $T = T(m,n)$ has diameter bounded independently of $n$ or $m$.
Thus, the $e_q(T)$ in the punchline \eqref{Dud2} can be bounded by $ \sup_{k \in T} d(k, T_q) $ for an equally spaced $T_q$,
with the result that $ e_q = O ( 2^{-2^q})$ and  
$$
    E \left[ \max_{e^{-m-1} \le k/n  \le e^{-m}} (S_k - E S_k )  \right] \le c''',
$$
where $c'''$ can be chosen fixed for all $n$ and $m$. 

It is left is to demonstrate that  $\sup_{k> c \log n} E S_k$ is similarly bounded. Since $E \log (Z_k/z_k^o) \one_{Q^c}$ (for $(Z,z)$ either $(X,x)$ or $(Y, y)$) can made exponentially small in $n$,  
we have
$$
 E  \log  \left( \frac{Y_k/y_k^o}{X_k/ x_k^o } \right)  = E \left(   \frac{X_k - x_k^o}{x_k^o}  -  \frac{Y_k - y_k^o}{y_k^o}\right) 
 + 
 E  \left(   \frac{(X_k - x_k^o)^2}{2 (x_k^o)^2}  -  \frac{(Y_k - y_k^o)^2}{2 (y_k^o)^2}\right) + O (k^{-3/2}),
$$
after restoring the integrals to the full domain (from $Q$) and using that $E |Z_k -z_k^o |^3 = O (n^{-3/2} )$ and once more
that $z_k^o \ge c \sqrt{k/n}$ in the last term. But by Corollary \ref{prop:Shifts} we  have for example,
\begin{align*}
   \sum_{k = c \log n}^n  \left| E \left[   \frac{X_k - x_k^o}{x_k^o}  -  \frac{Y_k - y_k^o}{y_k^o} \right]  \right| & =   \sum_{k = c \log n}^n  \phi(k/n)^{-1}  |E (X_k - x_k^o) - (Y_k - y_k^o) |  (1+o(1)) \\
    & \le c'' \sum_{k = c \log n}^n (\log k)^5 k^{-3/2} = o(1),
\end{align*}
with a similar conclusion for the sum of the mean-square differences.
\end{proof}

\subsection{Near the singularity}

Proposition \ref{cheapprop} actually uses Proposition \ref{prop:kernelbound} as input, in addition to the next rough estimate. Again the strategy is similar to that in \cite{RR1} (see Sections 3.5-3.6 there).

\begin{lemma} 
\label{cheaplemma}
For any $c>0$ there exist events $\mathcal{B}_n$ of probability tending to one on  which
\begin{equation}
\label{cheapbound}
    K_n(s,t) \le \frac{c' }{\phi(t)}  \exp{ \left(  \kappa_n \int_s^t \frac{d \tau}{\phi({\tau})} \right)} \quad
  \mbox{ for }  0 \le s \le t \le c \frac{\log n}{n},
\end{equation}
with a constant $c' = c(V, \beta, a, c)$ and $\kappa_n = c' \sqrt{ n \log_2 n}$.
\end{lemma}

Granted this we will first prove the proposition and return to the proof of Lemma \ref{cheaplemma} afterwards.

\begin{proof}[Proof of Proposition \ref{cheapprop}]  As the integral in question in \eqref{vanishonstrip} is increasing in $c$ we may as well assume that it is large enough so that Proposition \ref{prop:kernelbound} is in place. In addition, we will invoke that proposition in the following way.  Denote by 
$\mathcal{A}_n$ the event that the random sequence $C_n$ appearing \eqref{ineq:kernelbound} exceeds some $c'$. By choice of $c'$ we can take the probability of $\mathcal{A}_n$ as close to one as we like. We are left to show that the appraisal \eqref{vanishonstrip} takes place on the intersection of $\mathcal{A}_n$ and $\mathcal{B}_n$. 

Now denoting  $\delta_n  = c \frac{\log n}{n}$ and on $\mathcal{B}_n$,  Lemma \ref{cheaplemma} gives that, 
\begin{align}
\label{firstintegral}
  \iint_{  0 \le s \le t  \le \delta_n}  | K_n(s,t)|^2 ds dt  & \le  c' \int_0^{\delta_n} \int_0^t \frac{1}{t} e^{\kappa_n (\sqrt{t} - \sqrt{s} )} ds dt \\
    & = 2 c' \int_0^{\delta_n}  \left( \frac{e^{\kappa_n \sqrt{t}} - 1}{\kappa_n^2 t} - \frac{1}{\kappa_n \sqrt{t}} \right) dt \nonumber \\
    & \le c' \, \frac{e^{\kappa_n \sqrt{{\delta_n}}}}{\kappa_n^2}, \nonumber
\end{align}
which is $O(\frac{1}{n^{1- \eta}})$ for any $\eta> 0$. The first inequality uses \eqref{cheapbound} along with the fact that $\phi(t)$ is bounded above and below by constant multiples of $\sqrt{t}$ for $t \ll 1$ (and absorbs these constants into an adjusted $c'$). The second inequality is an elementary Laplace approximation.

On the remaining domain of integration write
$$
   \int_{\delta_n}^1 \int_0^{\delta_n} | K_n(s,t)|^2 ds dt = \int_{\delta_n}^1 | K_n(\delta_n, t)|^2  dt  
    \int_0^{\delta_n} X_{[n \delta_n]}^2 | K_n(s, \delta_n)|^2 ds,
$$
which, restricted to the event $\mathcal{A}_n \cap \mathcal{B}_n$, is bounded above as in
\begin{align}
\label{secondintegral}
   \int_{\delta_n}^1 \int_0^{\delta_n} | K_n(s,t)|^2 ds dt  & \le c'  \int_{\delta_n}^1 |K_{\epsilon}( \delta_n, t) |^2  dt \,  
   \int_0^{\delta_n} e^{\kappa_n (\sqrt{\delta_n} - \sqrt{s}) } ds  \\
   & \le c' \left( \delta_n^{a-2 \epsilon} \int_{\delta_n}^1 t^{-(a+1+2 \epsilon)} dt \right)  \frac{e^{\kappa_n \sqrt{\delta_n}}}{\kappa_n^2}. \nonumber
\end{align}
The first line employs the same arguments used in the first line of \eqref{firstintegral}, as well as the observation that the proof of Lemma \ref{cheaplemma}  includes the bound $ X_{[n \delta_n]} K_n(s, \delta_n) \le c' e^{\kappa_n \int_s^t \frac{d \tau}{\phi({\tau})}}$ on the event in question.
The second line recalls the definition of $K_{\epsilon}$ from \eqref{boundingkernel} and again that uses the largest contribution comes from of the origin where $\theta(t)$ and $\phi^2(t)$ are bounded in terms of $t$. Bounding the remaining integral gives that \eqref{secondintegral} is controlled by a constant multiple of $  { \delta_n^{-4 \epsilon}  \kappa_n^{-2}} {e^{\kappa_n \sqrt{\delta_n}}}$, which tends to zero like a small negative power of $n$
by choosing  $\epsilon>0$ small enough.
\end{proof}

It remains to go back and establish Lemma \ref{cheaplemma}.

\begin{proof}[Proof of Lemma \ref{cheaplemma}]  The
events $\mathcal{B}_n$ are constructed so that  the inequality 
\begin{equation}
\label{dumbineq}
    \frac{Y_k}{X_k} \le 1 + c' \frac{ \sqrt{{\log_2 n}}}{ \sqrt{ n}  \phi(k/n)}
\end{equation}
holds for all indices $k \le c \log n$  with a fixed constant $c' = c'(c, \beta , a , V)$. Granted this one has that 
$$
   X_{ \lfloor n t \rfloor} K_n(s,t) = \prod_{k =  \lfloor ns \rfloor }^{ \lfloor nt \rfloor -1} \frac{Y_k}{X_k} \le \exp
    \left( c' \sqrt{n \log_2 n} \sum_{k =  \lfloor ns \rfloor }^{ \lfloor nt \rfloor -1}  \frac{1}{n \phi(k/n)}    \right), 
$$
for $s, t \le c \frac{\log n}{n}$ , simply due to $(1 +a ) \le e^a$ for $a\ge 0$. Further estimating above the obvious Riemann sum produces the exponential factor in the advertised \eqref{cheapbound}. An appropriate upper  bound on the  $(X_{\lfloor n t \rfloor})^{-1}$ prefactors will follow in the coarse of establishing 
\eqref{dumbineq}.

To begin, using Proposition \ref{prop:Concentration} yet again we have that
\begin{equation}
\label{round11}
    P \left(  | (X_k, Y_k) - (x_k^o, y_k^o) | \ge c' \sqrt{ \frac{\log_2 n }{n} } \mbox{ for any } k \le c \log n  \right) \le c \log n \times (\log n)^{-\gamma(c')},
\end{equation}
where $\gamma$ can be made large by choice of $c'$. On the other hand we also have that
\begin{equation}
\label{round12}
   |x_k^o - \phi(k/n) | + | y_k^o - \phi(k/n)| \le c'   \frac{1}{\sqrt{n}},
\end{equation}
for $k \le c \log n$. The latter follows from the established $O( 1/\sqrt{ n} )$ closeness of the true and fine minimizers for $k \le c \log n$  (Proposition \ref{prop:FineMin}) coupled with the explicit formulas for the fine minimizers.
Now set 
$$
  \mathcal{B}_n^{(1)} = \Bigl\{    | (X_k, Y_k) - (\phi(k/n), \phi(k/n)) | \le c' \sqrt{ \frac{\log_2 n }{n}}  \mbox{ for } c'' \log_2 n \le k \le c \log n \Bigr\} 
$$
with a $c''$ to be chosen momentarily. We have just explained why $P( \mathcal{B}_n^{(1)}) = 1 - o(1)$, while on 
that event there is the bound
\begin{align}
\label{YX1}
   \frac{Y_k}{X_k}  & \le \frac{\phi(k/n) + c'  \sqrt{ \frac{\log_2 n }{n}}  }{\phi(k/n) - c'  \sqrt{ \frac{\log_2 n }{n}}  }  \\
   &
    \le 1 + 4  c'  \frac{ \sqrt{\log_2 n }}{ \sqrt{n} \phi(k/n)}, \quad  \mbox { granted that }  c'  \frac{ \sqrt{\log_2 n }}{ \sqrt{n} \phi(k/n)} < \frac{1}{2},  \nonumber
\end{align}
which can be guaranteed by taking $c''$ large depending on $c'$. This is \eqref{dumbineq}, after a readjustment of $c'$.

Moving to the range $k \le c'' \log_2 n$ first observe that for such indices the right hand side  of \eqref{dumbineq} can be replaced by a  constant multiple of 
$\sqrt{\frac{\log_2 n}{k}}$. Here we select a small   $\delta > 0$ such that
\begin{equation}
\label{delta_sandwhich}
      \delta  \sqrt{\frac{k}{n}}  \le  \phi(k/n) \le \frac{1}{\delta}   \sqrt{\frac{k}{n}}  \quad \mbox{ for } k \le c'' \log_2 n,
\end{equation}
and define 
\begin{align*}
  \mathcal{B}_n^{(2)} =  & \Bigl\{    | Y_k -  \phi(k/n)) | \le c' \sqrt{ \frac{\log_3 n }{n}}   \mbox{ for } k \le c'' \log_2 n \Bigl\} \\
                                       & \cap \Bigl\{    X_k > \delta^2 \phi(k/n)  \mbox{ for } \log_4 n \le k \le c'' \log_2 n, \,  X_k >  \frac{1}{\sqrt{n \log_3 n}}  
                                         \mbox{ for }  k \le \log_4 n \Bigr\}.                                                                                   
\end{align*}
With the corresponding restrictions on $X_k$ and $Y_k$ in place it holds that
\begin{equation}
\label{YX2}
   \frac{Y_k}{X_k} \le \frac{1}{\delta^2} +  \frac{c'}{\delta} \sqrt{\frac{\log_3 n}{k}}, \quad  \mbox{ for } k \in [ \log_4 n, c'' \log_2 n],
\end{equation}
while
\begin{equation}
\label{YX3}
     \frac{Y_k}{X_k} \le  \left( \frac{1}{\delta} + c'  \right) \log_3 n, \quad \mbox{ for } k \in [ 1, \log_4 n]. 
\end{equation}
The right hand sides of both \eqref{YX2} and \eqref{YX3} are then $O \left( \sqrt{\frac{\log_2 n}{k}} \right)$ as desired.

Leaving aside the verification that $ P(\mathcal{B}_n^{(2)}) = 1 - o(1)$, the claim is that the proof is complete by choosing 
 $ \mathcal{B}_n =  \mathcal{B}_n^{(1)}  \cap \mathcal{B}_n^{(2)}$. The remaining detail is the prefactor $[X_{ \lfloor n t \rfloor}]^{-1}$
multiplying $ \prod_{k =  \lfloor ns \rfloor }^{ \lfloor nt \rfloor -1} \frac{Y_k}{X_k}$ in the definition of the kernel. For $k =  \lfloor n t \rfloor \ge \log_4 n$ 
the definition of $\mathcal{B}_n$ explicitly restricts $[X_k]^{-1}$ to be less than a constant multiple of $[ \phi(k/n) ]^{-1}$, as desired. For 
smaller values of $k$ the bound available from the definition of $\mathcal{B}_n^{(2)}$ is off by a factor of $\sqrt{ \frac{\log_3 n}{k}}$. But this is readily absorbed into the upper bound on $Y_k / X_k$ provided by \eqref{YX3}. 

%The $\log_4 n$ cutoff is somewhat arbitrary.   
%Anything tending to infinity more slowly than $\log_3 n $ would do.

Returning to the probability of $\mathcal{B}_2^{(n)}$, that
$$
  P \left(  | Y_k -  \phi(k/n)) | \ge c' \sqrt{ \frac{\log_3 n }{n}}   \mbox{ for } k \le c'' \log_2 n \right) = o(1)
  %= O \left( \log_2 n \times (\log_2 n)^{-\gamma'(c')} \right),
$$
%with a $\gamma'$ made large with $c'$ 
holds by the same reasoning behind  \eqref{round11} and \eqref{round12}. 
The twist  is the different type of restriction placed on $X_k$ from below  for $k$ in this range, the lower bound provided by Gaussian concentration now  being cumbersome.

For $k \in [\log_4 n, c'' \log_2 n]$ we require an upper bound on 
$$
   P \left( X_k \le \delta^2 \phi(k/n) \right) \le P \left( X_k \le \delta \sqrt{k/n} \right).
$$
As Lemma \ref{chi_bound} below shows,  this probability  is less than a constant multiple of $P( \zeta \le \epsilon E \zeta )$ in which  $\zeta \sim \chi_{\beta(k+a)}$ and 
$\epsilon = \epsilon(\delta, a , \beta)$  can be taken less than $ 1/2$  granted that $\delta \ll 1$. Here we  use that $k \gg 1$ and that for any $\chi_r$ random variable $ \sqrt{ r - 1/2 } \le E \chi_r \le  \sqrt{r}$ as long as $r \ge 1$. Next bring in the following tail inequality: again with $\chi_r$ denoting a random variable of the indicated law and $r \ge 1$,
\begin{equation}
\label{chiLSI}
      P ( | \chi_r - E \chi_r | \ge \eta E \chi_r) \le 2 e^{-\eta^2 r/2}. 
\end{equation}
This is a consequence of the Logarithmic Sobolev Inequality for measures with strictly log-concave densities (Chapter 5 of \cite{L}) 
along with the  mentioned upper/lower bounds on $E\chi_r$. Combining these remarks the conclusion is that
$$
  P \Bigl(  X_k \le \delta^2 \phi(k/n), \, \mbox{ for some } k \in [ \log_4 n, c'' \log_2 n] \Bigr) \le c''' \sum_{k \ge \log_4 n} e^{-\frac{\beta}{8} k},
$$
which tends  to zero as $n \rightarrow \infty$.

 Finally, for $ k \le \log_4 n$ (with the real problem being when $k$ is order one) the  inequality \eqref{chiLSI}  becomes ineffective,  but we get by with the more elementary  $P ( \chi_r \le \delta) \le \kappa \delta^r $ where $\kappa $ is fixed (for $r \ge1$). This simple estimate yields
  $$
   P \left( X _k \le \frac{1}{\sqrt{ n \log_3 n}},  \mbox{ for some } k \le \log_4 n \right) \le 
   \sum_{k =1}^{\log_4 n} \left( \frac{c'''}{\sqrt{\log_3 n}} 
   \right)^{\beta(k+a)}  \rightarrow 0,
 $$
 after another application of Lemma \ref{chi_bound}.
\end{proof}

\begin{lemma}
\label{chi_bound}
Let $k \le  c \log n$ and denote by $I_k$ the interval $[(k-d) \vee 0, k+d]$. Then, for large enough $n$, 
\begin{equation}
\label{eq:chi_bound}
  P \left ( X_k \le t \, | \,  X_j, Y_j, \, j \in I_k \right) {\one}_{ \{ X_j,  Y_j  \le c \sqrt{ \frac{ \log_2 n}{n}}, \,   j \in I_k  \}} \le  c' P  (\zeta \le c' \sqrt{n} t ), 
\end{equation}
in which $\zeta$ is a $\chi_{\beta (k+a)}$ random variable and $c'$ depends on  $c$ (and $\beta, a, V$).
\end{lemma} 

\begin{proof} First, conditioned on all the other variables $X_k$ has density function proportional to,
$$
   f(x; z) =  x^{\beta (k+a)-1} e^{-n \gamma x^2 + n  \Gamma(x, z)}, 
$$
where $\gamma >0$ is a constant and $\Gamma(x,z)$ is polynomial in $x$ and the other coordinates $z_j \in I_k$. Note that the exponent of any variable in $\Gamma$ is at least two. Hence, with $p(t)$ denoting the left hand side of \eqref{eq:chi_bound} we have that 
 \begin{align*}
 p(t)  & \le   \frac{ \int_0^t f(x;z) dx}{ \int_0^{ \frac{\log n}{\sqrt{n}}} f(x; z) dx }   {\one}_{ \{ x, z \le  c \sqrt{ \frac{ \log_2 n}{n}} \}}
           \le \frac{ \int_0^t x^{\beta (k+a)-1}    e^{-  n \gamma  ( 1 - c'' \frac{\log_2 n}{n} ) x^2 } dx } 
                        {   \int_0^{\frac{\log n}{\sqrt{n}}}   x^{\beta (k+a)-1}    e^{-  n \gamma  ( 1 + c'' \frac{\log^2 n}{n} ) x^2}  dx},
 \end{align*}
 with a constant $c''$ depending on $V$ and $\beta$. 
 
 Next denote $c_{+} = \gamma ( 1 + c'' \frac{\log^2 n}{n} )$ and $c_{-}  = \gamma (1- c'' \frac{\log_2 n}{n} )$. For large enough $n$ there are the bounds $ \gamma/2 \le c_{-}, c_{+} \le 2 \gamma$,  while $(c_{+}/c_{-} )^r \le 2$ for  any exponent $r$ of order $ \log n$.
Changing variables then produces 
 $$
 p(t) \le 2   \frac{ \int_0^{ \sqrt{2 \gamma n } t } x^{\beta (k+a)-1}    e^{-x^2 } dx }
                        {   \int_0^{ \sqrt{\gamma/2} \, {\log n}}   x^{\beta (k+a)-1}    e^{-x^2}  dx}.                                   
$$
The proof is completed by a standard stationary phase calculation which shows that the integral in the denominator can be bounded below (independently of
of $n$, $k \le n$, $\beta$ or $a$) by a constant multiple of $\int_0^{\infty}   x^{\beta (k+a)-1}    e^{-x^2}  dx$, the regular $\chi_{\beta (k+a)}$  normalizer.
\end{proof}

\section*{Appendix}

We include here the derivation that our matrix model $B(X,Y) B(X,Y)^{T}$ (with $(X,Y)$ sampled from the measure $P$) realizes the joint eigenvalue density \eqref{eigdensity}. To simplify notation a bit we take $c \prod_{i < j} |\lambda_i - \lambda_j |^{\beta} \prod_{i=1}^n \lambda_i^{\gamma} e^{-V(\lambda_i)}$ as the target density, with any $\gamma > -1$ and polynomial $V$.

Also, to make a more direct connection with the derivation of the $\beta$-Laguerre ensemble one finds in the literature (in say \cite{DE}) consider first an upper bidiagonal matrix $M$ with coordinates labeled in decreasing order: $M_{i,i} = x_{n-i+1}$ for $i=1, \dots, n$ and 
$M_{i, i+1} = y_{n-i}$ 
for $i =1,\dots, n-1$, with all $x_i$,  $y_i$ positive. Also introduce the tridiagonal coordinates through a Jacobi matrix
$T = T(a, b)$ with $T_{i,i} = a_{n-i+1}$ for $i=1, \dots, n$ and $T_{i, i+1} = T_{i+1, i} = b_{n-i}$ for $i = 1, \dots, n-1$. Here each 
$a_i \in \RR$ and each $b_i \in \RR_+$. We track the calculation from eigenvalue/eigenvector coordinates to $(x,y)$ coordinates via $ Q \Lambda Q^{\dagger} = T = M M^T$. Here $Q$ is the eigenvector matrix, of which we only need the first 
components. These can be chosen to be real positive, and are denoted $(q_1, \dots, q_{n-1})$, noting that $q_n$ is specified by $\sum_{i=1}^n q_i^2 = 1$.

Next, we have that the Jacobians for the maps from $(\lambda, q)$ to $(a,b)$, and then from $(a,b)$ to $(x,y)$
are given by
$$
  J = q_n \frac{\prod_{i=1}^n q_i}{\prod_{i=1}^{n-1} b_i},  \quad  J' = 2^n x_1 \prod_{i=2}^{n} x_i^2,
$$
respectively. See  \cite[eq.~1.156]{FBook} for the former. The latter is derived from the identities
$a_i = x_i^2 + y_{i}^2$ and $b_i = x_{i+1} y_{i}$ (where $y_n = 0$ is understood). 
We will also need the well-known relation,
$$
\prod_{i <j} (\lambda_i -\lambda_j)^2 =  \frac{\prod_{i=1}^{n-1} b_i^{2i} }{\prod_{i=1}^{n} q_i^2},
$$
for which see  \cite[eq.~1.148]{FBook}.

Since we obviously have that $\sum_{i=1}^n V(\lambda_i) = \tr V(M M^T)$, the necessary computation is:
\begin{align*}
 ( \prod_{i=1}^n \lambda_i )^{\gamma}  \prod_{i < j} |\lambda_i - \lambda_j |^{\beta} ( q_n^{-1} \prod_{i=1}^n q_i^{\beta-1} ) \, d q \wedge d \lambda
& =   (  \prod_{i=1}^n x_i^{2 \gamma} ) 
\left(     \frac{\prod_{i=1}^{n-1} (x_{i+1} y_i)^{\beta i}}{\prod_{i=1}^{n-1} q_i^{\beta}} \right) J J' \, dx \wedge dy \\
& = 2^n  \prod_{i=1}^{n} x_i^{2 \gamma + \beta (i-1) +1}                 
 \prod_{i=1}^{n-1} y_i^{\beta i -1}   dx \wedge dy.
\end{align*}
Putting in $\gamma = \frac{\beta}{2}(a+1) - 1$ we recognize the factors in $x_i^{ \beta (a+i) -1}$ and $y_i^{\beta i -1}$ in the claimed bidiagonal matrix density \eqref{thelaw}. Here we have decided to work with $B = S M S^{-1}$ where $S$ is the antidiagonal matrix of alternating signs. This transformation does not effect the joint density of the individual coordinates, and 
one has that the eigenvalues of $B B^T$ and $M M^T$ agree.

\bigskip

\noindent
{\bf{Acknowledgements.}}  B.R.~was supported in part by NSF grant DMS-1340489. It is a pleasure to thank Manjunath Krishnapur and Michel Ledoux for several helpful discussions.

\end{document}